\documentclass{amsart}

\usepackage{graphicx}
\usepackage{amsmath}
\usepackage{mathrsfs}
\usepackage{amsmath, amscd}
\usepackage{amssymb}
\usepackage{tikz}
\usepackage{tikz-cd}
\usepackage{enumitem}
\newtheorem{theorem}{Theorem}[section]
\newtheorem{lemma}[theorem]{Lemma}
\newtheorem{proposition}[theorem]{Proposition}
\newtheorem{corollary}[theorem]{Corollary}

\theoremstyle{definition}
\newtheorem{definition}[theorem]{Definition}
\newtheorem{example}[theorem]{Example}

\newtheorem{construction}[theorem]{Construction}

\newtheorem*{convention}{Convention}
\theoremstyle{remark}
\newtheorem{remark}[theorem]{Remark}
\newtheorem{remarkthm}{Remark}[theorem]
\numberwithin{equation}{section}

\newcommand{\Z}{\mathbb{Z}}

\begin{document}

\title[On the Cohomology of $BPU_{n}$]{On the Cohomology of the Classifying Spaces of Projective Unitary Groups}

\author{Xing Gu}

\address{School of Mathematics and Statistics, The University of Melbourne, Parkville VIC 3010, Australia}
\email{xing.gu@unimelb.edu.au}
\thanks{The author was supported by the Australian research council via grant ARC-DP170102328.}





\begin{abstract}
 Let $BPU_{n}$ be the classifying space of $PU_n$, the projective unitary group of order $n$, for $n>1$. We use a Serre spectral sequence to determine the ring structure of $H^{*}(BPU_{n}; \mathbb{Z})$ up to degree $10$, as well as a family of distinguished elements of $H^{2p+2}(BPU_{n}; \mathbb{Z})$, for each prime divisor $p$ of $n$. We also study the primitive elements of $H^*(BU_n;\mathbb{Z})$ as a comodule over $H^*(K(\mathbb{Z},2);\mathbb{Z})$, where the comodule structure is given by an action of $K(\mathbb{Z},2)\simeq BS^1$ on $BU_n$ corresponding to the action of taking the tensor product of a complex line bundle and an $n$ dimensional complex vector bundle.
\end{abstract}

\maketitle
\tableofcontents

\textbf{Acknowledgements}
This paper is a revision of the first chapter of the author's PhD thesis at the University of Illinois at Chicago. The author would like to express his gratitude to his co-advisor Benjamin Antieau who proposed this subject and offered helpful advice, to his thesis advisor Brooke Shipley for her constant encouragement, financial support and helpful advice, and to Benedict Williams, who inspired the author with his work on this subject. In addition, the three persons mentioned above have read several earlier versions of this paper and shared their valuable opinions with the author. The author would also like to thank Masaki Kameko for his hospitality during the author's short visit to Saitama, Japan, and for a helpful conversation in which he presented many ideas related to the computation of the cohomology of $BPU_{n}$. The author is in debt to Aldridge Bousfield for examining earlier versions of this paper, during the course of which he found several problems and proposed solutions. The author would like to thank Diarmuid Crowley, Christian Haesemeyer and Stephane Dartois for helpful conversations. Finally, the author thanks the referee for a very informative report.

\section{Introduction}\label{Introduction}
Let $U_{n}$ be the unitary group of order $n$, and consider the unit circle group $S^1$ of complex numbers as the normal subgroup of scalars of $U_{n}$. The quotient group, denoted hereafter by $PU_{n}$, is called the projective unitary group of order $n$. Its classifying space $BPU_{n}$ is a topological space determined by $PU_{n}$ up to homotopy type, with a canonical base point, characterised by the fact that for a well-behaved topological space $X$ with a base point, the set of pointed homotopy classes of maps, $[X, BPU_{n}]$, has a natural one-to-one correspondence with the isomorphism classes of principal $PU_{n}$ bundles. Since $BPU_n$ is homotopy equivalent to $BPGL_n$, the classifying space of the projective general linear group $PGL_n$, which is the automorphism group of the algebra of $n\times n$ matrices, the set $[X, BPU_{n}]$ also classifies
bundles of $n\times n$ matrix algebras over $X$, known as the topological Azumaya algebras of degree $n$ over $X$.

The space $BPU_n$ plays a fundamental role in the study of the topological period-index problem (Antieau and Williams \cite{An}, \cite{An1}, and Gu \cite{Gu}, \cite{Gu1}), a topological analog of the period-index problem concerning elements of Brauer groups and their representing Azumaya algebras (\cite{Co}, Colliot-Th{\'e}l{\`e}ne, 2002). The space $BPU_n$ also arises in the study of Dai-Freed anamolies in partical physics (Garc{\'i}a-Etxebarria and Montero, \cite{Ga}).

The cohomology rings of $BPU_n$ with coefficients in $\mathbb{Z}/p$ for various primes $p$, as well as its Brown-Peterson cohomology, for $n$ of various forms, have been considered by Kameko (\cite{Ka}, 2008), Kono and Mimura (\cite{Ko}, 1975), Kono, Mimura and Shimada (\cite{Ko1}, 1975), Toda (\cite{To}, 1987), and Vavpeti{\u c} and Viruel (\cite{Va}, 2004).

The subject of this paper is the integral cohomology $H^{*}(BPU_{n};\mathbb{Z})$. This has been studied for $n=p$ a prime number: When $p=2$ we have $PU_2\simeq SO_3$ and it follows from \cite{Br} (Brown, 1982). When $p$ is odd, it is studied by Vistoli (\cite{Vi}, 2009). For an arbitrary integer $n$, we know $H^{k}(BPU_{n};\mathbb{Z})$ for $k\leq 5$ (\cite{An}, Antieau and Williams, 2014). To the author's knowledge the torsions of $H^*(BPU_{n};\mathbb{Z})$ for an arbitrary $n$ has not been considered before.

Our first theorem is as follows. The notations in the statement are made clear in the sequel.

\begin{theorem}\label{maincalc}
 For an integer $n>1$,  $H^{*}(BPU_n;\mathbb{Z})$ in degrees $\leq 10$ is isomorphic to the following graded ring:
 $$\mathbb{Z}[e_2, \cdots, e_{j_n}, x_1, y_{3,0}, y_{2,1}]/I_n.$$
 Here $e_i$ is of degree  $2i$, $j_n=\textrm{min}\{5, n\}$. The degrees of $x_1, y_{3,0}, y_{2,1}$ are $3, 8, 10$, respectively. The ideal $I_n$ is generated by
 \[nx_1,\quad \epsilon_2(n)x_1^2,\quad \epsilon_3(n)y_{3,0},\quad \epsilon_2(n)y_{2,1},\]
 \[\delta(n)e_{2}x_1,\quad (\delta(n)-1)(y_{2,1}-e_{2}x_1^2),\quad e_{3}x_1,\]
 where $\epsilon_p(n)=\operatorname{gcd}(p,n)$, and
 \begin{equation*}
 \delta(n)=
 \begin{cases}
  2, \textrm{ if }n=4l+2\textrm{ for some integer }l,\\
  1, \textrm{ otherwise}.
 \end{cases}
 \end{equation*}
\end{theorem}

We outline the strategy of the proof of Theorem \ref{maincalc} as follows. Consider the short exact sequence of Lie groups
\[1 \rightarrow S^1 \rightarrow U_{n} \rightarrow PU_{n} \rightarrow 1.\]
Applying the classifying space functor $B$, we obtain the following fiber sequence
\[BS^1 \rightarrow BU_{n} \rightarrow BPU_{n}.\]
Notice that $BS^1$ is the Eilenberg-Mac Lane space $K(\mathbb{Z},2)$. Therefore we have the following fiber sequence:
\begin{equation}\label{EMfibration}
U:BU_{n} \rightarrow BPU_{n}\xrightarrow{\chi} K(\mathbb{Z},3).
\end{equation}
Let $^{U}E^{*,*}_*$ be the integral cohomological Serre spectral sequence induced by (\ref{EMfibration}). This is the main object of interest in this paper.

Similarly, let $T^{n}$ and $PT^{n}$ be the maximal tori of $U_{n}$ and $PU_{n}$ respectively, and we have a fiber sequence
\begin{equation}\label{Torusfibration}
T:BT^n \rightarrow BPT^n \rightarrow K(\mathbb{Z},3).
\end{equation}
Let $^{T}E^{*,*}_{*}$ be the integral cohomological Serre spectral sequence associated to it.

Finally we let $^{K}E^{*,*}_*$ be the integral cohomological Serre spectral sequence associated to the path fibration
\begin{equation*}
K:K(\mathbb{Z},2) \rightarrow PK(\mathbb{Z},3) \rightarrow K(\mathbb{Z},3),
\end{equation*}
where $PK(\mathbb{Z},3)$ is the (contractible) path space (with one end fixed) of $K(\mathbb{Z},3)$. We compare $^{U}E^{*,*}_*$ with the other two spectral sequences via homological algebra of differential graded algebras and the theory of Chern classes, in particular the splitting principle.

As we shall see in Section \ref{1<k<6}, the map $BPU_{n}\xrightarrow{\chi} K(\mathbb{Z},3)$ in the fiber sequence (\ref{EMfibration}) represents an additive generator of $H^3(BPU_{n};\mathbb{Z})\cong\mathbb{Z}/n$. It is well-known (and revisited in Section \ref{DG Alg}) that for each prime $p$, the $p$-torsion element of $H^*(K(\mathbb{Z},3);\mathbb{Z})$ of the lowest dimension lives in dimension $2p+2$, which we denote by $y_{p,0}$. In consistency with Theorem \ref{maincalc}, we omit $\chi^*$, the notation for the induced homomorphism of cohomology rings, when there is no risk of ambiguity. In particular, in the following theorem we do not distinguish $y_{p,0}$ from $\chi^*(y_{p,0})$.

\begin{theorem}\label{2p+2}
Let $p$ be a prime. In $H^{2p+2}(BPU_{n};\mathbb{Z})$, we have $y_{p,0}\neq 0$ of order $p$ when $p|n$, and $y_{p,0}=0$ otherwise. Furthermore, the $p$-torsion subgroup of $H^k(BPU_n;\mathbb{Z})$ is $0$ for $3<k<2p+2$.
\end{theorem}
\begin{remarkthm}
When $p=2$, we have $y_{2,0}=x_1^2$.
\end{remarkthm}

When $n=p$ is an odd prime, the essential content of this theorem is implied by work of Vistoli (Theorem 3.4, \cite{Vi}, 2009).

To state the next theorem, we recollect some generality on the topology of classifying spaces of compact Lie groups. Let $G$ be a compact Lie group, $\Gamma$ a closed subgroup of its center. Let $P: BG\rightarrow B(G/\Gamma)$ be the canonical map. Then $P$ fits in a fiber sequence
\[BG\xrightarrow{P}B(G/\Gamma)\rightarrow B^2\Gamma,\]
which induces an action of $\Omega B^2\Gamma\simeq B\Gamma$ on $BG$:
\begin{equation}\label{fiberactiion}
\mu:B\Gamma\times BG\rightarrow BG.
\end{equation}
Let $R$ be a discrete commutative unital ring such that the K{\"u}nneth formula reduces to an isomorphism
\[H^*(B\Gamma\times BG;R)\cong H^*(B\Gamma;R)\otimes_R H^*(BG;R).\]
Then
\[\mu^*: H^*(BG;R)\rightarrow H^*(B\Gamma;R)\otimes_R H^*(BG;R)\]
makes $H^*(BG;R)$ a comodule over $H^*(B\Gamma;R)$. We denote the subring of primitive elements of $H^*(BG;R)$ (the elements $x$ such that $\mu^*(x)=1\otimes x$) by $PH(BG;R)$.

It is well known that we have $\operatorname{Im}P^*\subseteq PH^*(BG;R)$ (Proposition 3.8 of \cite{To}, Toda, 1987). It is therefore natural to ask when the equality holds.

Back to our case of $G=U_n$, $\Gamma=S^1$ and $G/\Gamma=PU_n$, the action
\[K(\mathbb{Z},2)\times BU_n\rightarrow BS^1\times BU_n\rightarrow BU_n\]
corresponds to the action of taking the tensor product of a complex line bundle and an $n$ dimensional complex vector bundle. Toda (\cite{To}, 1987) showed $\operatorname{Im}P^*=PH^*(BU_n;\mathbb{Z}/2)$ for $n\equiv 2\pmod 4$ and $n=4$. Our third theorem concerns the case of integral cohomology:
\begin{theorem}\label{primitive}
For $n>1$ and $k\leq 12$, we have $\operatorname{Im}(H^k(P;\mathbb{Z}))=PH^k(BU_n;\mathbb{Z})$.
\end{theorem}
The author is in debt to Diarmuid Crowley for introducing the notion of primitive elements, which leads to the present formulation of Theorem \ref{primitive}.

In Section \ref{DG Alg}, we study the cohomology of $K(\mathbb{Z},3)$ and the spectral sequence $^{K}E^{*,*}_*$. This is an example of the general theory of differential graded algebras and their bar constructions, developed in \cite{Ca} (Cartan and Serre, 1954-1955) and reviewed in the appendix.

We set up the apparatus that computes the differentials of ${^UE}^{*,*}_*$ in Section \ref{The Higher Diff}. We compare ${^UE}^{*,*}_*$, ${^TE}^{*,*}_*$ with ${^KE}^{*,*}_*$ to give all the differentials of $^{T}E^{*,*}_*$ and show that they detect considerably many differentials of ${^UE}^{*,*}_*$ via the splitting principle of complex vector bundles (Chapter 16, \cite{Sw}, Switzer, 1975).

We collect some auxiliary results and present a proof of Theorem \ref{2p+2} in Section \ref{Auxiliary}. Some of results in this section are devoted to simplify the proof of Theorem \ref{maincalc}, which occupies Section \ref{1<k<6}, \ref{k=7,8}, and \ref{k=9,10}.

In Section \ref{On Primitive} we discuss the primitive elements of $H^*(BU_n;\mathbb{Z})$ which enable us to prove Theorem \ref{primitive}. This also offers another way, in addition to the one discussed in Section \ref{The Higher Diff}, to study the differentials of $^UE_*^{*,*}$.

\section{DG Algebras and the Path Fibration of $K(\mathbb{Z},3)$}\label{DG Alg}
In \cite{Ca} (Cartan and Serre, 1954-1955) a differential graded algebra is constructed to calculate the integral homology of $K(\Pi,n)$ for any positive integer $n$ and a finitely generated abelian group $\Pi$. This approach in particular gives us the cohomology of $K(\mathbb{Z},3)$. However, for the purpose of this paper we need the differentials of the integral Serre spectral sequence of the fiber sequence
\begin{equation}\label{EMfib}
K(\mathbb{Z},2)\rightarrow PK(\mathbb{Z},3) \rightarrow K(\mathbb{Z},3)
\end{equation}
which requires information on the level of chain complexes rather then just cohomology. To obtain such information we consider an acyclic multiplicative construction (\cite{Ca}) $(A(2), A(3), M(3))$ which models the fiber sequence above.

Throughout the rest of this section, we write $E_{R}(x;k)$, $P_{R}(y;l)$, $E(x;k)$, $P_{R}(y;l)$ for exterior algebras and divided power algebras with coefficients in a ring $R$, or $\mathbb{Z}$, and with one generator of a specified degree. All tensor products are over the base ring $\mathbb{Z}$ unless otherwise specified.

In \cite{Ca}, Cartan and Serre constructed differential graded algebras $A(n)$ such that the cohomology of $A(n)$ is isomorphic to that of $K(\Pi,n)$. We use the terminologies such as ``admissible words'' and their ``heights'' as well as the symbols $\sigma$, $\varphi$ and $\gamma_p$ as in \cite{Ca}, and proceed to elaborate this construction for $\Pi=\mathbb{Z}$ and $n=2$ and $3$. The reader may refer to the appendix for a short survey on relevant definitions and results in \cite{Ca}.

In the case $n=2$, the admissible words of height $2$ are of the form $\gamma_k(\sigma^2)$ for any nonnegative integer $k$. Therefore we take $A(2)=P(u,2)$ where $u$ corresponds to the operation $\sigma^{2}$. Since $A(2)$ is $0$ in odd degrees, the differential is $0$.

In the case $n=3$, the admissible words of height $3$ are $a_{p,k}=\varphi_{p}\gamma_{p^k}\sigma^{2}, b_{p,k}=\sigma\gamma_{p^{k+1}}\sigma^{2}$, and $b_1=\sigma^{3}$ for nonnegative integers $k$ and prime numbers $p$. It follows that we have
\[A(3)=\bigotimes_{p}\big[\bigotimes_{k\geq 0}P(a_{p,k}; 2p^{k+1}+2)\otimes_{\mathbb{Z}} E(b_{p,k}; 2p^{k+1}+1)\big] \otimes_{\mathbb{Z}}
  E(b_1; 3).\]
Here the numbers after the semicolons are the degrees of the elements preceding the semicolons. The differential is defined by
\begin{equation}
\begin{cases}
  \bar{d}(\gamma_l(a_{p,k}))=pb_{p,k}\gamma_{l-1}(a_{p,k}),\\
  \bar{d}(b_{p,k})=0, \bar{d}(b_1)=0
\end{cases}
\end{equation}

We proceed to take $M(3)=A(2)\otimes_{\mathbb{Z}} A(3)$ as a graded ring, and define its differential $d$ in such a way that it makes $M(3)$ acyclic, and when passing to $A(3)=\mathbb{Z}\otimes_{A(2)}M(3)$, it induces $\bar{d}$, the differential of $A(3)$. The role of $M(3)$ is explained in Remark \ref{Serrefiltration}.

\begin{lemma}[Lucas' Theorem]\label{Lucas}
 Let $p$ be a prime number, $k=\sum_{r=0}^{n}k_{r}p^r$, and $l=\sum_{r=0}^{n}l_{r}p^r$ such that $0\leq k_{r}, l_{r}
 \leq p-1$ are integers, and $k_n\neq 0$. Then
 \[{k \choose l}\equiv \prod_{r=0}^{n}{k_r \choose l_r} \mod p.\]
 Here ${i \choose j}=0$ for $i<j$.
\end{lemma}

\begin{proof}
 For an independent variable $w$ we have
 $$(1+w)^{k}=\prod_{r=0}^{n}(1+w)^{k_{r}p^{r}}\equiv \prod_{r=0}^{n}(1+w^{p^r})^{k_{r}} \mod p.$$
 The result is verified by comparing the coefficient of $w^{l}$ on both sides of the equation above.
\end{proof}

\begin{lemma}\label{gcd}
 For a prime $p$ and $1\leq i\leq k+1$ the greatest common divisor of $\frac{(p^{k+1}-1)!}{(p^{k+1}-p^i)!}$ and $\frac{p^{i}!}{p^{i-1}}$ is $(p^{i}-1)!$.
\end{lemma}
The proof is elementary and left to the reader.

In particular there are integers $\{\lambda_{i}^{p, k+1}, \mu_{i}^{p, k+1}\}_{i=0}^{k}$ such that
\begin{equation}\label{eq:Lambdas}
\frac{1}{(p^{k+1}-1)!p^i}=\lambda_{i}^{p, k+1}\frac{1}{(p^{k+1}-1)!p^{i-1}}+\mu_{i}^{p, k+1}\frac{1}{(p^{k+1}-p^{i})!p^{i}!}.
\end{equation}

\begin{definition}\label{defM(3)}
 We define $M(3)$ as follows:
 \begin{enumerate}
  \item Let
  \begin{equation}
  M(3)_k=\sum_{i+j=k}A(2)_i\otimes_{\mathbb{Z}} A(3)_j.
  \end{equation}

 \item The differential of $M(3)$ is determined by the following:
  \begin{equation*}
  \begin{split}
  &d(u)=0; \quad d(b_1)=u; \quad d(b_{p,k})=\gamma_{p^{k+1}}(u);\\ &d(a_{p,k})=(pb_{p,k}-\Lambda_{0}^{p, k+1}b_{1}\gamma_{p^{k+1}-1}(u))-\sum_{i=1}^{k}\Lambda_{i}^{p, k+1}b_{p,i-1}\gamma_{p^{k+1}-p^{i}}(u)
  \end{split}
  \end{equation*}
  where $\{\Lambda_{i}^{p, k+1}\}_{i=0}^{k}\subset \mathbb{Z}$ are defined as follows.

  Fix a set of integers $\{\lambda_{i}^{p, k+1}, \mu_{i}^{p, k+1}\}_{i=0}^{k}$ as in Lemma \ref{gcd}. Define
  \begin{enumerate}
  \item
  \[\Lambda_{0}^{p,k+1}=\lambda_{1}^{p, k+1}\cdots \lambda_{k}^{p, k+1}.\]
  \item
     \[\Lambda_{i}^{p,k+1}=\mu_{i}^{p, k+1}\lambda_{i+1}^{p, k+}\cdots \lambda_{k}^{p, k+1}, \textrm{ if } i=1,\cdots, k-1.\]
  \item
     \[\Lambda_{k}^{p,k+1}=\Lambda_{k}^{p, k+1}=\mu_{k}^{p, k+1}.\]
  \end{enumerate}
  Lemma \ref{gcd} ensures that $d$ is indeed a differential.
 \end{enumerate}
  As the tensor product of two cochain complexes, there is a bi-degree on $M(3)$. More precisely, let $B\otimes C$ be a monomial in $M(3)$ such that
  \[B\in A(3);\ C\in A(2).\]
  Then the bi-degree of $B\otimes C$ is $(s,t)=(\operatorname{deg}(B), \operatorname{deg}(C))$, and $s$ induces a filtration $F_K$ on $M(3)$:
  \[F_K^{s}=\{\sum_iB_i\otimes C_i | B_i\in A(3), C_i\in A(2), B_i\textrm{ has degree }\leq s.\}.\]

\end{definition}

The following result plays a central role in this section.

\begin{proposition}\label{M(3)}
 $M(3)$ is acyclic.
\end{proposition}
Before presenting the proof, the following remarks are in order:
\begin{remark}\label{Serrefiltration}
It follows from \cite{Br1} ((4.1) Theorem and (7.2) Corollary) that the spectral sequence associated to the filtration $F$ on $M(3)$ is isomorphic to the Serre spectral sequence of the path space fibration
\begin{equation*}
K: K(\mathbb{Z},2)\rightarrow PK(\mathbb{Z},3)\rightarrow K(\mathbb{Z},3),
\end{equation*}
for cohomology, namely $^KE_*^{*,*}$ defined in Section 1, with coefficients in $\mathbb{Z}$. The case for the Serre spectral sequence for homology is similar.
\end{remark}

\begin{remark}
The reason that we give the unpleasantly complicated formulae in Definition \ref{defM(3)} is as follows. Remark \ref{Serrefiltration} indicates that $^KE_*^{*,*}$ is the spectral sequence associated to the filtered chain complex $M(3)$. Therefore, at least in principle, all the differentials of $^KE_*^{*,*}$ can be determined mechanically, though the process might be tedious. This would be useful if one would like to explore the full potential of the spectral sequence $^UE_*^{*,*}$, perhaps with a computer program. The formulae are of little use though, if one is only interested in low dimensional computation.
\end{remark}

For the proof of Proposition \ref{M(3)}, recall the following lemmas:
\begin{lemma}[Cartan and Serre, \cite{Ca}]\label{eledga}
 For any integer $q$ and a prime number $p$ the following DGA over $\mathbb{Z}/p$ is acyclic:
 \[P_p(a; 2pq+2)\otimes_{\Z} E_p(b; 2q+1)\otimes_{\Z} \mathbb{Z}/p[c; 2q]/c^p\]
 with differential given by
 $$d(c)=0;\ d(b)=c;\ d(\gamma_{l}(a))=c^{p-1}b\gamma_{l-1}(a),\ l\geq 1.$$
\end{lemma}

\begin{proof}
 In fact, a chain homotopy can be defined on linear generators as follows:
 $$c\rightarrow b;\ c^{p-1}b\gamma_{l-1}(a)\rightarrow \gamma_{l}(a),\ l\geq 1$$
 and all the other linear generators are sent to $0$.
\end{proof}

Lemma \ref{eledga} immediately generalize to the following

\begin{lemma}\label{M(3)[p]}
 Let $M(3)[p]$ be the DGA
 \[\bigotimes_{k\geq 0}\big[P(a_{p,k}; 2p^{k+1}+2)\otimes_{\Z} E(b_{p,k}; 2p^{k+1}+1)\big]\otimes_{\Z} E(b_1; 3)\otimes_{\Z} P(u;2)\otimes_{\Z} \mathbb{Z}/p\]
 with differential $d[p]$ defined as follows:
 \begin{equation*}
  \begin{split}
  &d[p](u)=0; \quad d[p](b_1)=u; \quad d[p](b_{p,k})=\gamma_{p^{k+1}}(u);\\ &d[p](\gamma_{l}(a_{p,k}))=-\Lambda_{k}^{p,k+1}b_{p,k-1}\gamma_{p^{k+1}-p^{k}}(u)\textrm{ if }k>0,\\
  &d[p](\gamma_{l}(a_{p,0}))=-\Lambda_{k}^{p,1}b_{1}\gamma_{p-1}(u).
  \end{split}
  \end{equation*}
 Then $M(3)[p]$ is acyclic.
\end{lemma}

\begin{proof}
We rewrite $M(3)[p]$ as
 \begin{equation*}
 M(3)[p]=\bigotimes_{k\geq 0}P(a_{p,k}; 2p^{k+1}+2)\otimes_{\Z} E(b_{p,k-1};2p^{k}+1)\otimes_{\Z}\mathbb{Z}[\gamma_{p^k}(u)]/(\gamma_{p^k}(u)^p)\otimes_{\Z}\mathbb{Z}/p,
 \end{equation*}
 where $b_{p,-1}=b_{1}$.
 Notice that $M(3)[p]$ is a tensor product of DGA's of the form in Lemma \ref{eledga} indexed by $k$, and the result follows.
\end{proof}

\begin{proof}[Proof of Proposition \ref{M(3)}]
 We proceed to show that $M(3)\otimes_{\Z} \mathbb{Z}/p$ is acyclic for all primes $p$. Since $M(3)$ is a degree-wise finitely generated free abelian group, this, together with the K{\"u}nneth formula proves the proposition.

 Let ($\tilde{E}_{s,t}^{r}[p], d_{s,t}^{r})$ be the spectral sequence associated to the filtration $F_{K}$ on $M(3)\otimes_{\Z}\mathbb{Z}/p$. Then obviously $E_{s,t}^{0}[p]=E_{s,t}^{1}[p]$ and $d_{*,*}^{1}$ is as follows:
 \[d^1(u)=0;\ d^1(b_1)=0;\ d^1(b_{p,k})=0;\]
 \[
    d^1(\gamma_{l}(a_{p',k}))=
    \begin{cases}
     p'\gamma_{l-1}(a_{p',k})b_{p',k}, \quad \textrm{if }p'\neq p\\
     0,\quad \textrm{if }p'=p
    \end{cases}
 \]
 for any prime $p'$. Therefore we have
 \begin{equation}
 \begin{split}
  \tilde{E}_{*,*}^{2}[p]\cong\bigotimes_{k\geq 0}\big[P_{p}(a_{p,k}; 2p^{k+1}+2)\otimes_{\Z} E_{p}(b_{p,k}; 2p^{k+1}+1)\big] \otimes_{\Z} E_{p}(b_1; 3)\otimes_{\Z} P_{p}(u; 2)\\
  \cong \bigotimes_{k\geq 0}\big[P_{p}(a_{p,k}; 2p^{k+1}+2)\otimes_{\Z} E_{p}(b_{p,k-1}; 2p^{k}+1) \otimes_{\Z} \mathbb{Z}/p[\gamma_{p^k}(u)]/\gamma_{p^k}(u)^p\big]
 \end{split}
 \end{equation}
 where $b_{p,-1}=b_1$. We proceed to consider all the higher differentials on $\tilde{E}_{*,*}^{2}[p]$. Notice that all the non-trivial differentials on $P_{p}(a_{p,k}; 2p^{k+1}+2)\otimes_{\Z} E_{p}(b_{p,k}; 2p^{k}+1) \otimes_{\Z} \mathbb{Z}/p[\gamma_{p^k}(u)]/\gamma_{p^k}(u)^p$ are the following:
 \begin{equation*}
 \begin{split}
   &d^{2(p^{k+1}-p^{k})+1}(a_{p,k})=\Lambda_{k}^{p, k+1}b_{p,k-1}\gamma_{p^k}(u);\\
   &d^{2(p^{k+1}-p^{k})+1}(b_{p,k-1})=\gamma_{p^k}(u);\\
   &d^{2(p^{k+1}-p^{k})+1}(\gamma_{p^k}(u))=0.
 \end{split}
 \end{equation*}
 By definition $\Lambda_{k}^{p, k+1}=\mu^{p,k+1}_{k}$. By Lemma \ref{gcd} $p\lambda^{p,k+1}_{i}+\mu^{p,k+1}_{i}{p^{k+1}-1\choose p^i-1}=1$, which shows that $\Lambda_{k}^{p, k+1}=\mu^{p,k+1}_{k}$ is invertible mod $p$.

 Notice that same statement of the definition of the filtration $F_K$ in (4) of Definition \ref{defM(3)} can be applied to define a filtration on $M(3)[p]$. A direct comparison shows that this filtration induces a spectral sequence which is identical to $\tilde{E}_{*,*}^{2}[p]$ after the $E_2$-page. The theorem then follows from Lemma \ref{M(3)[p]}.
\end{proof}

\begin{remark}
 The construction of $M(3)$ and the proof of Proposition \ref{M(3)} are inspired by exp.9 and exp.11 of \cite{Ca} (Cartan and Serre, 1954-1955).
\end{remark}

\begin{corollary}\label{higherdiff}
 Let $\tilde{E}_{*,*}^*$ be the spectral sequence associated to the filtration $F_K$ on $M(3)$. Then the higher differentials satisfy
\begin{equation*}
 \begin{split}
   &d_{3}(b_1)=u;\\
   &d_{2p^{k+1}+1}(b_{p,k})=\gamma_{p^{k+1}}(u),\quad k\geq 0;\\
   &d_{r}(\gamma_{i}(u))=0,\quad \textrm{for all }r, i,
 \end{split}
\end{equation*}
together with the Leibniz rule. In particular, we have
\begin{equation*}
d_{2p^{k}+1}(b_{p,i_m}\cdots b_{p,i_2}b_{p,i_1})=b_{p,i_{m-1}}\cdots b_{p,i_2}b_{p,i_1}\gamma_{p^{i_m+1}}(u),
\end{equation*}
for $i_1\geq i_2\geq\cdots\geq i_m\geq0$.
\end{corollary}
\begin{proof}
This is straightforward computation, as the differential of $M(3)$ is described explicitly in Definition \ref{M(3)}.
\end{proof}
\begin{remark}
We have not discussed the differentials of
\[b_{p,I}:=\frac{1}{p}d(a_{p,i_m}\cdots a_{p,i_2}a_{p,i_1})=\sum_{j=1}^{r}(-1)^{\mu_j}a_{p,i_m}\cdots b_{p,i_j}\cdots a_{p,i_1},
\]
where $I=(i_m,\cdots,i_2,i_1)$ is an ordered sequence for nonnegative integers $i_m\leq\cdots\leq i_1$, and
\begin{equation}\label{muDefinition}
\mu_j=\sum_{k>j}i_k, \mu_m=1.
\end{equation}
This is somewhat complicated and not needed in this paper.
\end{remark}

To study the cohomology, we will need a diagonal approximation
\[
  \Delta: A(n)\rightarrow A(n)\otimes_{\Z} A(n)
\]
for $n=2$ and $3$. We take $\Delta$ to be the (unique) homomorphism of DGA's with divided power operations such that for an individual admissible word $x$ the following is satisfied:
\begin{equation}\label{Deltax}
\Delta(x)=
\begin{cases}
a_{2,k}\otimes1+b_{2,k-1}\otimes b_{2,k-1}+1\otimes a_{2,k}, \textrm{ if }x=a_{2,k}, n=3, k\geq 1,\\
a_{2,0}\otimes1+b_1\otimes b_1+1\otimes a_{2,0}, \textrm{ if }x=a_{2,0}, n=3,\\
x\otimes 1+1\otimes x,  \textrm{ otherwise.}
\end{cases}
\end{equation}

Consider the following homomorphism induced by $\Delta$:
\[\Delta_p: A_p(n)\rightarrow A_p(n)\otimes_{\Z} A_p(n),\]
where $A_p(n)=A(n)\otimes_{\Z}\mathbb{Z}/p$. As a morphism of $\mathbb{Z}/p$-modules this is simply $\Delta\otimes_{\Z}\mathbb{Z}/p$. However, for $n=3$ and $p=2$ there are some complications on the divided power algebra, as we are about to observe.
\begin{lemma}\label{Deltap}
For $n=2,3$ and any prime $p$, $\Delta_p$ is the (unique) homomorphism of DGA's with divided power operations such that for an individual admissible word $x$ it satisfies
\[\Delta_p(x)=x\otimes 1+1\otimes x.\]
\end{lemma}
\begin{proof}
 For $p\neq2$ there is nothing to show. Let $\bar{a}_{2,k}$, $\bar{b}_{2,k}$ be the image of $a_{2,k}$, $b_{2,k}$ in $A_2(3)$ under the obvious projection. Then it suffices to verify that
 \[\Delta_2(\bar{a}_{2,k})=\bar{a}_{2,k}\otimes1+\bar{b}_{2,k-1}\otimes \bar{b}_{2,k-1}+1\otimes \bar{a}_{2,k}\]
 can be derived from the characterization of $\Delta_2$ in the lemma. Indeed, we have the relation given in exp. 10, \cite{Ca}:
 \[\varphi_2=\gamma_2\cdot\sigma,\]
 which implies
 \begin{equation}
 \begin{cases}
 \bar{a}_{2,k}=\gamma_2(\bar{b}_{2,k-1}),k\geq 1,\\
 \bar{a}_{2,0}=\gamma_2(\bar{b}_1).
 \end{cases}
 \end{equation}
The lemma then follows from the above, together with the product formula for divided power operations (exp. 7, \cite{Ca}).
\end{proof}
It then follows from exp.9, exp.10 of \cite{Ca} that for $n=2,3$ and each prime $p$, $\Delta_p$ induced the cup product over the cohomology ring $H^*(K(\mathbb{Z},n);\mathbb{Z}/p)$. Therefore $\Delta$ is a diagonal approximation.

We proceed to study the dual complex of $A(2)$, $A(3)$ and $M(3)$, which we denote by
\begin{equation*}
\begin{cases}
  C(2)^{i}=\operatorname{Hom}(A(2)_{i}, \mathbb{Z}),\\
  C(3)^{i}=\operatorname{Hom}(A(3)_{i}, \mathbb{Z}),\\
  W(3)^{i}=\operatorname{Hom}(M(3)_{i}, \mathbb{Z}).
\end{cases}
\end{equation*}
Let $y_{p,k}, x_{p,k}, x_1, v$ be the dual of $a_{p,k}, b_{p,k}, b_1, u$, respectively. Then it follows from (\ref{Deltax}) that we have
\begin{equation}\label{xyrelation}
\begin{cases}
y_{2,k}=x_{2,k-1}^2, k\geq 1,\\
y_{2,0}=x_1^2.
\end{cases}
\end{equation}

The $\mathbb{Z}$-algebra structures of $C(2)$ and $C(3)$ induced by the diagonal approximations $\Delta$ follows readily:
\begin{equation}\label{products}
\begin{split}
C(2)&=\mathbb{Z}[v;2],\\
C(3)&=\bigotimes_{p\neq2}\big[\bigotimes_{k\geq 0}\mathbb{Z}[y_{p,k}; 2p^{k+1}+2]\otimes_{\Z} E(x_{p,k}; 2p^{k+1}+1)\big]\\
&\otimes_{\Z}\bigotimes_{k\geq 0}\mathbb{Z}[x_{2,k}; 2^{k+2}+1]\otimes_{\Z}\mathbb{Z}[x_1; 3],\\
W(3)&=C(2)\otimes_{\Z} C(3).
\end{split}
\end{equation}
Again the numbers following the semicolons indicate degrees. We proceed to consider the differentials of the cochain complexes above. For $C(2)$ it follows readily that the differential is trivial. For $C(3)$ it follows from (\ref{Deltax}) and (\ref{xyrelation}) that we have
\begin{equation}\label{C3diff}
\begin{cases}
d_1(x_{p,k})=py_{p,k}, d_1(y_{p,k})=d_1(x_1)=0,\textrm{for $p$ a prime and }k\geq 0,\\
d_1(\alpha\beta)=d_1(\alpha)\beta+(-1)^{\operatorname{deg}(\alpha)}\alpha d_1(\beta),
\end{cases}
\end{equation}
where $\alpha$, $\beta$ take the forms of monomials in $x_1$, $x_{p,k}$ and $y_{q,l}$ for various $p,q,k,l$, such that the exponents of $x_1$ and $x_{2,k}$ ($k\geq0$) in $\alpha\beta$ are at most $1$. The relation (\ref{xyrelation}) ensures that the differential $d_1$ is uniquely determined by the above.
\begin{remark}
It is an observation due to Ben Antieau that $W(3)$ is not a DGA. Indeed, the differential $\delta$ of $W(3)$ does not satisfy the Leibniz rule. Consider $v^2$, the dual of $\gamma_{2}(u)$. More precisely, the pairing $\langle v^2, \xi\rangle$ is characterized by
 \begin{equation*}
 \langle v^2, \xi\rangle=
 \begin{cases}
 k, \quad \xi=k\gamma_{2}(u)\\
 0, \quad \textrm{otherwise.}
 \end{cases}
 \end{equation*}
To find $\delta(v^{2})$, notice that $M(3)_5$ is generated as an abelian group by $b_{2,1}$ and $b_{1}u$, on which the differential acts as $d(b_{2,1})=\gamma_{2}(u)$ and $d(b_{1}u)=u^2=2\gamma_{2}(u)$. Since $b_{1}$ and $u$ are elements in terms of the coproduct on $M(3)$, the dual of $b_{1}u$ is $x_{1}v$. Thus, the relation $\langle v^2, d(-)\rangle=\langle \delta(v^2), -\rangle$ implies $\delta(v^2)=2vx_{1}+x_{2,1}\neq 2vx_{1}$, which disproves the Leibniz formula. However the Leibniz formula holds for the differentials cohomological Serre Spectral sequences.
\end{remark}
The cohomology ring $H^3(K(\mathbb{Z},3);\mathbb{Z})$ can now be described in terms of generators and relations. It is easier to do so by describing its torsion free component, the $p$-primary components for all primes $p$, and their relations. Whenever it is obvious from the context, we do not distinguish a cocycle in $C(3)$ and the cohomology class it represents. Let $_pA$ denote the $p$-primary summand of a ring $A$ (possibly without unit) of which the underlying abelian group is finitely generated. For $R$ a (graded-commutative) ring without unit, let $\hat{R}$ be the (graded anti-commutative) unital ring $R\oplus\mathbb{Z}$, where the summand $\mathbb{Z}$ lives in degree $0$ and has a generator acting as the unit of the ring.
\begin{proposition}\label{K(Z,3)general}
The graded ring structure of $H^*(K(\mathbb{Z},3);\mathbb{Z})$ is described in terms of generators of and relations in $C(3)$ as follows:
\[H^*(K(\mathbb{Z},3);\mathbb{Z})\cong[\bigotimes_{p}\widehat{_pH^*}(K(\mathbb{Z},3);\mathbb{Z})]
\otimes_{\Z}\mathbb{Z}[x_1]/(y_{2,0}-x_1^2),
\]
where $p$ runs over all prime numbers, and $_pH^*(K(\mathbb{Z},3);\mathbb{Z})$
is the graded-commutative $\mathbb{Z}/p$-subalgebra (without unit) of $C_3\otimes_{\Z}\mathbb{Z}/p$, generated by the elements of the form
\begin{equation}\label{pgenerators}
  y_{p,I}:=\frac{1}{p}d_1(x_{p,i_m}\cdots x_{p,i_2}x_{p,i_1})=\sum_{j=1}^r(-1)^{\mu_j}x_{p,i_m}\cdots y_{p,i_j}\cdots x_{p,i_1},
\end{equation}
where $I=(i_m,\cdots,i_1)$ is an ordered sequence of nonnegative integers $i_m<\cdots <i_1$, and $\mu_j$ is defined as in (\ref{muDefinition}). For $I=(i)$, we have $y_{p,I}=y_{p,i}$.

Moreover, the classes $\{y_{p,i}\}_{i\geq 0}$ are algebraically independent.

Similarly, there is a set of generators for homology given by $b_{p,I}$ for nonnegative integers $i_m\leq\cdots\leq i_1$.
\end{proposition}
Beware that the elements $\{y_{p,I}\}_I$ are not algebraically independent. For example, fix $i_0>i_1>i_2\geq 0$, and let $I_0=(i_2,i_1)$, $I_1=(i_2,i_0)$ and $I_2=(i_1,i_0)$. Then one has the following relation:
\[y_{p,0}y_{p,I_0}-y_{p,1}y_{p,I_1}+y_{p,2}y_{p,I_2}=0.\]
\begin{proof}
The proof of the statement for cohomology and homology are the same and we only present the one for cohomology.

It follows from (\ref{C3diff}) that the differential of $C(3)$ localized at $p$ is trivial modulo $p$. Therefore the $\mathbb{Z}/p$-algebra $H^{*}(K(\mathbb{Z},3);\mathbb{Z}/p)$ is isomorphic to
\begin{equation}\label{K(Z,3)Z/p}
 \begin{cases}
 \big[\bigotimes_{k\geq 0}\mathbb{Z}/p[y_{p,k}; 2p^{k+1}+2]\otimes_{\Z} E(x_{p,k}; 2p^{k+1}+1)\big]\otimes_{\Z}\mathbb{Z}/p[x_1; 3]\textrm{ if } p>2,\\
 \bigotimes_{k\geq 0}\mathbb{Z}/2[x_{2,k}; 2^{k+2}+1]\otimes_{\Z}\mathbb{Z}/2[x_1; 3]\textrm{ if } p=2.
 \end{cases}
\end{equation}
On the other hand, for any prime $p$, a $p$-primary element of $H^*(K(\mathbb{Z},3);\mathbb{Z})$ is $p$-torsion. Therefore the mod $p$ reduction homomorphism
 \[H^*(K(\mathbb{Z},3);\mathbb{Z})\rightarrow H^*(K(\mathbb{Z},3);\mathbb{Z}/p)\]
is injective. To complete the proof, notice that the relation $y_{2,0}-x_1^2=0$ is given in (\ref{xyrelation}).
\end{proof}

It suffices in this paper to have the following
\begin{corollary}\label{K(Z,3)}
In degree less then $15$, $H^{*}(K(\mathbb{Z}, 3);\mathbb{Z})$ is isomorphic to the following ring:
 \[\mathbb{Z}[x_1; 3]\otimes_{\Z}\bigotimes_{k\geq 0; p}\mathbb{Z}/p[y_{p,k}; 2p^{k+1}+2]/(x_{1}^2-y_{2,0}),\]
 where $p$ runs over all prime numbers.
\end{corollary}

We proceed to apply this setup to study the Serre spectral sequence ${^KE}_{*}^{*,*}$. Let $F_K$ be the filtration on $W(3)$ obtained by dualizing the filtration $F$ on $M(3)$. Then the spectral sequence associated to $F_K$ is isomorphic to ${^KE}_{*}^{*,*}$ starting from the $E_2$-page (Remark \ref{Serrefiltration}).

In particular, we have
\begin{equation*}
{^KE}_{2}^{s,t}\cong H^s(K(\mathbb{Z},3);H^t(K(\mathbb{Z},2);\mathbb{Z}))\cong
\begin{cases}
H^s(K(\mathbb{Z},3);\mathbb{Z}), t \textrm{ is even},\\
0, t \textrm{ is odd}.
\end{cases}
\end{equation*}

The higher differentials of ${^KE}_{2}^{*,*}$ may be obtained by carefully writing down a formula of the dual of the differential of $M(3)$ as given in Definition \ref{M(3)}. Alternatively, we may consider the dual of the spectral sequence ${^KE}_{*,*}^*$. For example, the counterpart of Corollary \ref{higherdiff} is given in the following
\begin{corollary}\label{diff0}
 The higher differentials of ${^KE}_{*}^{*,*}$ satisfy
 \begin{equation*}
 \begin{split}
   &d_{3}(v)=x_1,\\
   &d_{2p^{k+1}-1}(p^{k}x_{1}v^{lp^{e}-1})=v^{lp^{e}-1-(p^{k+1}-1)}y_{p,k},\quad k\geq e, \operatorname{gcd}(l,p)=1,\\
   &d_{r}(x_1)=d_{r}(y_{p,k})=0,\quad \textrm{for all }r, k>0,
 \end{split}
\end{equation*}
and for $I=(i_m,\cdots,i_1)$, $0<k<i_m<\cdots <i_1$,
\begin{equation*}
 d_{2p^{k}+1}(y_{p,I}v^{p^{k}})=y_{p,I'},
\end{equation*}
where $I'=(k,i_m,\cdots,i_1)$, and the Leibniz rule. Here ${^KE}_{2p^{k+1}-1}^{3,2(lp^e-1)}$ is generated by $p^{k}v^{lp^{e}-1}x_1$.
\end{corollary}
This corollary covers all the differentials we need from this spectral sequence.
\begin{proof}
The first three equations follows by straightforward computation. To prove the last one, we consider the dual of the last equation of Corollary \ref{higherdiff}:
\begin{equation}\label{db_p,I}
d_{2p^{k}+1}(b_{p,i_m}\cdots b_{p,i_2}b_{p,i_1})=b_{p,i_{m-1}}\cdots b_{p,i_2}b_{p,i_1}\gamma_{p^{i_m+1}}(u),
\end{equation}
for $i_m<\cdots <i_1$. We make this precise as follows. For obvious degree reasons we have $\tilde{E}_r^{s,0}\subset\tilde{E}_2^{s,0}$ for $r>2$. We then apply (\ref{db_p,I}) to obtain the following
\[b_{p,i_m}\cdots b_{p,i_2}b_{p,i_1}\in\tilde{E}_{2p^{k}+1}^{s-1,0}\subset\tilde{E}_2^{s,0}
\]
for some $s$. Since $E_2^{s,0}$ are direct sums of cyclic groups of prime prders for $s>3$, we apply the universal coefficient theorem to obtain
\[^KE_2^{s,0}\cong\operatorname{Ext}_{\mathbb{Z}}^1(\tilde{E}^{s-1,0}_2,\mathbb{Z}),\]
and furthermore,
\[^KE_2^{s,0}\cong\operatorname{Ext}_{\mathbb{Z}}^1(\tilde{E}^{s-1,0}_2,\mathbb{Z})
\twoheadrightarrow\operatorname{Ext}_{\mathbb{Z}}^1(\tilde{E}^{s-1,0}_{2p^k+1},\mathbb{Z}),\]
where the last arrow is a splitting surjection. The differential
\[^Kd_{2p^k+1}^{s-2p^k-1,2p^k}: {^K}E_{2p^k+1}^{s-2p^k-1,2p^k}\rightarrow {^KE}_{2p^k+1}^{s,0}\]
is obtained by applying the functor $\operatorname{Ext}_{\mathbb{Z}}^1(-,\mathbb{Z})$ to the differential
\begin{equation*}
\begin{split}
d^{2p^k+1}_{s-1,0}: \tilde{E}^{2p^k+1}_{s-1,0}&\rightarrow \tilde{E}^{2p^k+1}_{s-2p^k-2,2p^k},\\
b_{p,k}b_{p,i_m}\cdots b_{p,i_2}b_{p,i_1}&\mapsto b_{p,i_m}\cdots b_{p,i_2}b_{p,i_1}\gamma_{p^{k+1}}(u)
\end{split}
\end{equation*}
and the result follows.
\end{proof}
Figure \ref{KE} shows a low dimensional picture.
\begin{remark}
In principle, Corollary \ref{diff0} may be proved via studying how the differential of $W(3)$ behaves with respect to the dual filtration of $F$. But this could be complicated.
\end{remark}
\begin{figure}\label{KE}
\begin{tikzpicture}
\draw[step=1cm,gray, thick] (0,0) grid (8,6);
\draw [ thick, <->] (0,7)--(0,0)--(9,0);

\node [below] at (3,0) {$3$};
\node [below] at (6,0) {$6$};
\node [below] at (8,0) {$8$};
\node [left] at (0,2) {$2$};
\node [left] at (0,4) {$4$};
\node [left] at (0,6) {$6$};
\node [right] at (9,0) {$s$};
\node [above] at (0,7) {$t$};

\node [above right] at (0,0) {$\mathbb{Z}$};
\node [above right] at (0,2) {$\mathbb{Z}$};
\node [above right] at (0,4) {$\mathbb{Z}$};

\node [above right] at (3,0) {$\mathbb{Z}$};
\node [above right] at (3,2) {$\mathbb{Z}$};
\node [above right] at (3,4) {$\mathbb{Z}$};

\node [above right] at (6,0) {$\mathbb{Z}/2$};
\node [above right] at (6,2) {$\mathbb{Z}/2$};
\node [above right] at (6,4) {$\mathbb{Z}/2$};

\node [above right] at (8,0) {$\mathbb{Z}/3$};
\node [above right] at (8,2) {$\mathbb{Z}/3$};
\node [above right] at (8,4) {$\mathbb{Z}/3$};

\draw (0,2) [thick, ->] to (3,0);
\draw (3,2) [thick, ->] to (6,0);
\draw (0,4) [thick, ->] to  node [above] {$\times2$} (3,2);
\draw (3,4) [thick, ->] to (8,0);
\draw (0,6) [thick, ->] to node [above] {$\times3$} (3,4);
\end{tikzpicture}\\
\caption{Some non-trivial differentials of ${^KE}_{*}^{*,*}$.}
\end{figure}

\section{The Higher Differentials of $^{U}E_{*}^{*,*}$}\label{The Higher Diff}
In this section we identify some higher differentials of the spectral sequence $^{U}E_{*}^{*,*}$ which make it possible to prove Theorem \ref{maincalc}. To do so, we compare it to $^{K}E_{*}^{*,*}$, via the following commutative diagram:
\begin{equation}\label{diag}
\begin{CD}
  BS^1 @>\simeq >> K(\mathbb{Z},2)        @>>>   *                @>>>  K(\mathbb{Z},3)\\
  \Phi:         @.        @VV B\varphi V      @VVV                        @|\\
                @.        BT^{n}           @>>> BPT^{n}     @>>>   K(\mathbb{Z},3)\\
  \Psi:         @.          @VVB\psi V                         @VVB\psi' V    @|\\
                @.        BU_{n}           @>>> BPU_{n}    @>>> K(\mathbb{Z},3)\\
  \end{CD}
\end{equation}
Here $\varphi : S^1\rightarrow T^{n}$ is the diagonal, and $\psi$ and $\psi'$ are the inclusions of the maximal tori.
Then for any compact Lie group $G$ and its maximal torus $T$, $H^{*}(BG; \mathbb{Q})$ is the sub-$\mathbb{Q}$-algebra of $H^{*}(BT; \mathbb{Q})$ stable under the action of the Weyl group. In our case, we have
\begin{equation}\label{splitPU}
 H^{*}(BPU_{n}; \mathbb{Q})\xrightarrow{\cong} H^{*}(BPT^{n}; \mathbb{Q})^{\mathbf{W}},
\end{equation}
where ${\mathbf{W}}$ is the Weyl group. (See \cite{Hs}, Hsiang, 1975.)

\begin{remark}
Theorem \ref{primitive} is equivalent to the assertion that the homomorphism (\ref{splitPU}) is a surjection in dimension $\leq 12$, with the coefficient ring $\mathbb{Q}$ replaced by $\mathbb{Z}$.
\end{remark}

Let $v$ be the multiplicative generator of $H^{*}(BS^1;\mathbb{Z})$, $v^i$ be the $i$th copy of $v$ in $H^{*}(B(S^1)^n;\mathbb{Z})\cong H^{*}((BT^{n};\mathbb{Z})$, and $c_k\in H^{*}(BU_{n};\mathbb{Z})$ be the $k$th universal Chern class. Moreover let $\sigma_k\in H^{*}(BT^{n};\mathbb{Z})$ be the $k$th elementary polynomial in $v_1, \cdots, v_n$. In this case we have the splitting principle, which asserts that the above argument applies to integral cohomology (\cite{Sw}, Switzer, 1975). More precisely, it says that $B\psi^*$ is the inclusion taking $c_k$ to $\sigma_k$, or in other words,
\[B\psi^{*}(\sum_{k=0}^{n}c_{k}w^{k})=\prod_{i=1}^{n}(1+v_{i}w)\]
where $w$ is a polynomial generator.

Our first result is the following

\begin{proposition}\label{diff1}
 The differential  $^{T}d_{r}^{*,*}$, is partially determined as follows:

\begin{equation}
 ^{T}d_{r}^{*,2t}(v_{i}^{t}\xi)={(B\pi_i)^{*}}({^Kd}_{r}^{*,2t}(v^{t}\xi)),
\end{equation}

where $\xi\in {^{T}E}_{r}^{0,*}$, a quotient group of $H^*(K(\mathbb{Z}, 3);\mathbb{Z})$, and $\pi_i: T^{n}\rightarrow S^1$ is the projection of the $i$th diagonal entry. In plain words, $^{T}d_{r}^{*,2t}(v_{i}^{t}\xi)$ is simply $^{K}d_{r}^{*,2t}(v^{t}\xi)$ with $v$ replaced by $v_i$.
\end{proposition}

The proof is straightforward since the Serre spectral sequence is functorial.

We proceed to study the differentials $^{T}d_{*}^{0,*}$ with domain in the leftmost column.
\begin{proposition}\label{diff2}
 \begin{enumerate}

 \item The differential $^{T}d_{3}^{0,t}$ is given by the ``formal divergence''

 \[\nabla=\sum_{i=1}^{n}(\partial/\partial v_i): H^{t}(BT^{n};R)\rightarrow H^{t-2}(BT^{n};R),\]
 in such a way that $^{T}d_{3}^{*,*}=\nabla(-)\cdot x_{1}.$ For any ground ring $R=\mathbb{Z}$ or $\mathbb{Z}/m$ for any integer $m$.

 \item The spectral sequence degenerates at ${{^T}E}^{0,*}_{4}$. Indeed, we have $^{T}E_{\infty}^{0,*}={^{T}E_{4}}^{0,*}=\operatorname{Ker}{^Td}_{3}^{0,*}=\mathbb{Z}[v_{1}-v_{n},\cdots, v_{n-1}-v_{n}]$.
 \end{enumerate}
\end{proposition}

\begin{proof}
 (1) is an immediate consequence of the Leibniz rule, chain rule, and Proposition \ref{diff1}.

 Given a polynomial $\theta(v_{1}, \cdots, v_{n})$, we change variables to rewrite it as $\tilde{\theta}(v_{1}-v_{n}, \cdots, v_{n-1}-v_{n}, v_{n})$.  \ref{splitPU} then tells us that $H^{*}(BPU_{n}; \mathbb{Q})$ is the free commutative ring generated by $\{v_{i}-v_{n}\}_{i=1}^{n-1}$. The equation
 \begin{equation*}
 \begin{split}
 &\qquad {^Td}_{3}(\theta(v_{1}, \cdots, v_{n}))\\
 &=\sum_{i=1}^{n}\frac{\partial}{\partial v_i}\tilde{\theta}(v_{1}-v_{n}, \cdots, v_{n-1}-v_{n}, v_{n})\\
 &=\frac{\partial}{\partial v_n}\tilde{\theta}(v_{1}-v_{n}, \cdots, v_{n-1}-v_{n}, v_{n})\\
 \end{split}
 \end{equation*}
 then implies that ${^Td}_{3}(\theta(v_{1}, \cdots, v_{n}))=0$ if and only if $\tilde{\theta}(v_{1}-v_{n}, \cdots, v_{n-1}-v_{n}, v_{n})$ is independent of $v_{n}$. For obvious degree reasons this is the only nontrivial differential over $\mathbb{Q}$, which implies that $H^{*}(BPT^n; \mathbb{Q})$, as a $\mathbb{Q}$-subalgebra of $H^{*}(BT^n; \mathbb{Q})$, is free commutative, generated by $\{v_{i}-v_{n}\}_{i=1}^{n-1}$. On the other hand, notice that the following homomorphism of Lie groups
 \begin{equation*}
 T^n\rightarrow S^1,
 \left( \begin{array}{ccc}
\lambda_{1} & 0 & \ldots \\
0 & \lambda_{2} & \ldots \\
\vdots & \vdots & \lambda_{n}
\end{array} \right)
 \mapsto\lambda_i\lambda_j^{-1}
 \end{equation*}
 factors through $PT^n$. This implies that $v_i-v_j\in H^{*}(BPT^n; \mathbb{Z})$, which proves (2).
\end{proof}

\begin{corollary}\label{diff4}
 $^{U}d_{3}(c_{k})=\nabla(c_{k})x_1=(n-k+1)c_{k-1}x_1.$
\end{corollary}

\begin{proof}
 \[^{U}d_{3}(c_k)={^Td}_{3}(\sigma_k)=\nabla(c_{k})=\sum_{i=1}^{n}\frac{\partial\sigma_{k}}{\partial v_i}=(n-k+1)c_{k-1}x_1.\]
\end{proof}
We will study the operator $\nabla$ in more details in Section \ref{On Primitive}.

We proceed to study $^{T}d_{r}^{*,*}$ for greater $r$.

Consider a map $\kappa: T^{n}\rightarrow PT^{n}\times S^1$ defined as follows:
\begin{equation*}
\mathbf{}
\left( \begin{array}{ccc}
\lambda_{1} & 0 & \ldots \\
0 & \lambda_{2} & \ldots \\
\vdots & \vdots & \lambda_{n}
\end{array} \right)
\mapsto
\Big(
\mathbf{}
\left[ \begin{array}{ccc}
\lambda_{1} & 0 & \ldots \\
0 & \lambda_{2} & \ldots \\
\vdots & \vdots & \lambda_{n}
\end{array} \right],
\lambda_{n} \Big)
\end{equation*}
where the matrix in the square bracket denotes its class in $PT^{n}$. Then $\kappa$ is an isomorphism of Lie groups, its inverse being the following:
\begin{equation*}
\Big(
\mathbf{}
\left[ \begin{array}{ccc}
\lambda_{1} & 0 & \ldots \\
0 & \lambda_{2} & \ldots \\
\vdots & \vdots & \lambda_{n}
\end{array} \right],
\lambda \Big)
\mapsto
\mathbf{}
\left( \begin{array}{ccc}
\lambda\lambda_{1}/\lambda_{n} & 0 & \ldots \\
0 & \lambda\lambda_{2}/\lambda_{n} & \ldots \\
\vdots & \vdots & \lambda
\end{array} \right).
\end{equation*}
Passing to classifying spaces, we have $B\kappa:BT^{n}\simeq BPT^{n}\times BS^{1}$.
\begin{proposition}\label{tensor}
 Let $\theta(v_{1}, \cdots, v_{n})\in H^{2t}(BT^{n};\mathbb{Z})$ be an element in $^{T}E_{2}^{0,2t}$. As in the proof of Proposition \ref{diff2} we apply the change of variable $(v_1-v_n,\dots,v_{n-1}-v_n,v_n)$ and rewrite $\theta$ as
 $$\theta=\sum_{i=0}^{t}\theta_{i}v_{n}^{i}$$
 such that $\theta_{i}\in \operatorname{Ker}\nabla$, $i=0,\cdots, t$. Then $$(B\kappa^{-1})^{*}(\theta)=\sum\theta_{i}\otimes v_{n}^{i}.$$
\end{proposition}
Consider the fiber sequence $BPT^{n}\rightarrow BPT^{n} \rightarrow *$ and let $^{P}E_{*}^{*,*}$ be its cohomological Serre Spectral sequence. We take the product of fiber sequences
\[(BPT^{n}\rightarrow BPT^{n} \rightarrow *)\times (BS^1\rightarrow * \rightarrow K(\mathbb{Z},3))\]
or equivalently
\[BPT^{n}\times BS^1\rightarrow BPT^{n}\rightarrow K(\mathbb{Z},3).\]
Then we have the following morphism between fiber sequences:
\begin{equation*}
\begin{CD}
    BT^{n}       @>>>   BPT^{n}               @>>>  K(\mathbb{Z},3)\\
          @VV B\kappa V      @|                        @|\\
    BPT^{n}\times BS^1          @>>> BPT^{n}     @>>>   K(\mathbb{Z},3)\\
\end{CD}
\end{equation*}
with the unspecified maps being the obvious ones. This diagram is easily seen to be commutative by the discussion above.

\begin{proposition}\label{diff3}
Let $^{P}E_{*}^{*,*}$ be the cohomological Serre Spectral sequence associated to $BPT^{n}\rightarrow BPT^{n} \rightarrow *$. Then the commutative diagram above induces an isomorphism of spectral sequences $^{T}E_{r}^{*,*}\cong {^{P}E}_{r}^{*,*}\otimes_{\Z} {^{K}E}_{r}^{*,*}$ for $r\geq 2$.
\end{proposition}
\begin{proof}
By Proposition \ref{tensor} the commutative diagram above induces an isomorphism of $E_3$ pages, and the statement follows from Theorem 3.4 in \cite{Mc} (McCleary, 2001).
\end{proof}

\begin{remark}
Let $\theta(v_{1}, \cdots, v_{n})\in H^{2t}(BT^{n};\mathbb{Z})$ represent an element in $^{T}E_{2}^{0,2t}$, $\xi$ be an element of $H^{s}(K(\mathbb{Z}, 3);\mathbb{Z})$ for some $s$. Suppose we have $\theta\xi\in {^{T}E}_{r}^{s,*}$. Applying the change of variables $(v_1-v_n,\dots,v_{n-1}-v_n,v_n)$ we rewrite $\theta$ as
 \[\theta=\sum_{i=0}^{t}m_{i}v_{n}^{i}\theta'_{i}\]
 where $m_{i}v_{n}^{i}\xi\in {^{T}E}_{r}^{*,*}$, $\theta'_{i}\in \operatorname{Ker}\nabla$, $i=0,\cdots, t$ for all $i$, such that $\theta'_{i}$ are primitive polynomials in the polynomial ring $\mathbb{Z}[v_1-v_n,\dots,v_{n-1}-v_n]$, i.e., only $\pm1$ divides all of its coefficients.  In particular, this is just the change of variable considered in Proposition \ref{tensor} where we have $\theta_{i}=m_{i}\theta'_{i}$. The differential ${^{T}d}_{r}^{3,*}$ is hence determined by
 \[{^{T}d}_{r}(\theta\xi)={^{T}d}_{r}(\sum_{i=0}^{t}m_{i}v_{n}^{i}\theta'_{i}\xi)
 =\sum_{i=0}^{t}{^{T}d}_{r}(m_{i}v_{n}^{i}\xi)\theta'_{i},\]
 where ${^{T}d}_{r}(m_{i}v_{j}^{i}\xi)$ is determined as in Proposition \ref{diff1}. This is how we compute the differentials in practice. This idea is explained more concretely in Example \ref{example}.
\end{remark}

Finally we are at the place to state our main result of this section, of which the proof is already clear.
\begin{proposition}\label{main}
Proposition \ref{diff1}, Proposition \ref{diff2} and Proposition \ref{diff3} determine all of the differentials of ${^TE}_{*}^{*,*}$. Restricted to symmetric polynomials in $v_1, \cdots, v_n$, they determine all of the differentials ${^Ud}_{r}^{s-r,t+r-1}$ of ${^UE}_{*}^{*,*}$ such that for any $r'<r$, ${^Td}_{r'}^{s-r',t-r'+1}=0$.
\end{proposition}
With enough patience one can apply this apparatus to calculate $H^{k}(BPU_{n}; \mathbb{Z})$ for many $k$ and $n$ up to group extension. The interested reader may take the following example as an exercise.
\begin{example}\label{example}
 Let $n=3$. We will show that $^{U}d_{3}(2c_{3}x_{1})=0$ but $2c_{3}x_{1}$ is not a permanent cocycle.
 \begin{proof}
 By the splitting principle (16.13 Corollary, \cite{Sw}, Switzer, 1975) the image of $c_3$ in $H^*(BT^3;\mathbb{Z})$ is $v_1v_2v_3$, and in $^{T}E_{*}^{*,*}$ we have
 \[2v_{1}v_{2}v_{3}x_1=2(v_{1}-v_{2})(v_{2}-v_{3})x_1+2[(v_{1}-v_{3})+(v_{2}-v_{3})]v_{3}^{2}x_1+2v_{3}^{3}x_1,\]
 the first two terms on the right side being easily checked to be permanent cocycles. For the third term we have
 \[^{T}d_{3}(2v_{3}^{3}x_1)=6v_{3}^{2}x_1^2=6v_{3}^{2}y_{2,0}=0\]
 since $2y_{2,1}=0$. This verifies the first statement. For the second one, we have
 \[^{T}d_{5}(2v_{3}^{3}x_1)=0\]
 since $v_{3}^3$ has order $4$ in $^{T}E_{4}^{*,*}$ but $y_{3,1}$ is of order $3$. However
 \[^{T}d_{7}(2v_{3}^{3}x_1)=y_{2,1}\neq 0,\]
 whence the second statement.
 \end{proof}
\end{example}
We end this section by several corollaries. Let $\nabla$, $\xi$ be the same as earlier, and let $\vartheta=\vartheta(c_1, \cdots, c_n)\in H^{*}(BU_{n};\mathbb{Z})$.

\begin{corollary}\label{cor1}
 $^{U}d_{3}(\vartheta\xi)={^{U}d}_{3}(\vartheta)\xi=\nabla(\vartheta)x_{1}\xi.$
\end{corollary}
This is immediate from the Leibniz rule and Proposition \ref{diff2}.

\begin{corollary}\label{cor2}
If $\vartheta\xi\in {^{U}E}_{r}^{s, r-1}$, or in other words, the image of $\rho\xi$ of $^{U}d_{r}$ lies in bottom row, then
\[^{U}d_{r}(\vartheta\xi)={^{K}d}_{r}(B(\psi\cdot\varphi)^{*}(\vartheta)\xi),\]
where $\varphi$ and $\psi$ are as in the diagram (\ref{diag}). In particular,
\[^{U}d_{r}(c_{k}\xi)={n\choose k}{^{K}d}_{r}(v^{k}\xi).\]
\end{corollary}

\begin{proof}
Consider the commutative diagram (\ref{diag}) at the beginning of this section. The result is immediate upon passing to the diagram of Serre spectral sequences: $\Psi\cdot\Phi$ induces $c_{k}\mapsto {n\choose k}v^k$ on the leftmost column and the identity on the rightmost one, the latter implying that the induced morphism of spectral sequences is the identity on $H^{*}(K(\mathbb{Z},3);\mathbb{Z})$ when restricted to the bottom rows of the $E_2$-page.
\end{proof}

\section{Auxiliary Results, Proof of Theorem \ref{2p+2}}\label{Auxiliary}
In the statement of Theorem \ref{maincalc}, the notations $x_1$, $y_{2,1}$ and $y_{3,0}$ are also used to denote elements in the cohomology group $H^*(K(\mathbb{Z},3);\mathbb{Z})$. This is on purpose, since as we shall see in the last three sections, $x_1$, $y_{2,1}$ and $y_{3,0}\in H^{*}(BPU_{n}; \mathbb{Z})$ are the images of their namesakes in
$H^*(K(\mathbb{Z},3);\mathbb{Z})$, under the homomorphism $H^{*}(K(\mathbb{Z},3); \mathbb{Z})\rightarrow H^{*}(BPU_{n}; \mathbb{Z})$, which is induced by the map $BPU_{n}\xrightarrow{\chi} K(\mathbb{Z}, 3)$ as in the fiber sequence (\ref{EMfibration}). We restate this fact as the following
\begin{corollary}\label{cohomology operations}
\begin{enumerate}
 \item
 When $n$ is even, the $2$-torsion subgroup of $H^{k}(BPU_{n}; \mathbb{Z})$ for $k=6, 9, 10$ are generated by the cohomology operations $y_{2,0}, y_{2, 0}x_{1}$, and $y_{2,1}$ respectively, applied to a generator of $H^{3}(BPU_{n}; \mathbb{Z})$.
\item When $3|n$, and $k=8$, the $3$-torsion subgroup of $H^{8}(BPU_{n}; \mathbb{Z})$ is generated by the cohomology operation $y_{3,0}$ applied to a generator of $H^{3}(BPU_{n}; \mathbb{Z})$.
\end{enumerate}
\end{corollary}

\begin{proof}
  In the statement of Theorem \ref{maincalc}, we see that for the $k$'s and $n$'s in the Corollary \ref{cohomology operations}, the corresponding torsion subgroups of $H^{k}(BPU_{n}; \mathbb{Z})$ are generated by the images of elements $y_{2,0}, y_{2, 0}x_{1},y_{2,1}$ and $y_{3,0}$ respectively, under the homomorphism $H^{k}(K(\mathbb{Z},3); \mathbb{Z})\xrightarrow{\chi^*} H^{k}(BPU_{n}; \mathbb{Z})$, and the proof is completed.
\end{proof}

We proceed to study the $p$-local components of $H^{*}(BPU_{n}; \mathbb{Z})$, for any prime number $p$.

\begin{lemma}\label{p-local}
If $t+1<p$, then
\[^{U}E_{\infty}^{3, 2t}\otimes_{\Z}\mathbb{Z}_{(p)}\cong {^{U}E}_{3}^{3, 2t}/\operatorname{Im}({^{U}d}_{3}^{0, 2t+2})\otimes\mathbb{Z}_{(p)}.\]
\end{lemma}

\begin{proof}
This follows since
\[^UE_{2}^{s,t}\cong H^s(K(\mathbb{Z},3);\mathbb{Z})\otimes_{\Z} H^t(BU_n;\mathbb{Z})\]
has no $p$-torsion for $s<2p+2$.
\end{proof}

It follows from Lemma \ref{p-local} that we have
\begin{proposition}\label{3<k<2p+1}
If $3<k<2p+1$ then $H^{k}(BPU_{n}; \mathbb{Z}_{(p)})$

\begin{equation*}
\cong
 \begin{cases}
 \operatorname{Ker}{^{U}d}_{3}^{0,k}\otimes_{\Z}\mathbb{Z}_{(p)}, \quad k \textrm{ is even,}\\
 H^{k-3}(BU_{n};\mathbb{Z}_{(p)})/\nabla H^{k-1}(BU_{n};\mathbb{Z}_{(p)}), \quad k \textrm{ is odd.}
 \end{cases}
\end{equation*}
\end{proposition}

We proceed to prove a lemma that simplifies the computation of the torsion free summand of $H^*(BPU_n;\mathbb{Z})$. We also present some immediate consequences of the computation apparatus developed in the previous section. At the end of this section we carefully decompose the statement of Theorem \ref{maincalc} to make its proof tractable.
\begin{lemma}\label{simplification}
Let $\bar{P}: BSU_{n}\rightarrow BPU_{n}$ be the map induced by the obvious quotient map $SU_{n}\rightarrow PU_{n}$.
\begin{enumerate}
\item
$\bar{P}$ induces an isomorphism $H^{*}(BPU_{n}; \mathbb{Q})\xrightarrow{\cong} H^{*}(BSU_{n};\mathbb{Q})$. In particular, $\bar{P}^{*}: H^{*}(BPU_{n}; \mathbb{Z})\xrightarrow{\cong} H^{*}(BSU_{n};\mathbb{Z})$ is a monomorphism modulo torsion, such that $\operatorname{Im}P^*$ has the same rank as $H^*(BPU_{n};\mathbb{Z})$ in each dimension.

\item
Let $m$ be an integer such that $\operatorname{gcd}(m,n)=1$, then $H^*(BPU_{n};\mathbb{Z})$ has no non-trivial $m$-torsion.
\end{enumerate}
\end{lemma}

\begin{proof}
 Let $SU_{n}$ be the special unitary group of degree $n$. The short exact sequence of Lie groups
 \[1\rightarrow\mathbb{Z}/n\rightarrow SU_{n}\xrightarrow{P} PU_{n}\rightarrow 1\]
 induces a fiber sequence
 \[K(\mathbb{Z}/n,1)\rightarrow BSU_{n}\rightarrow BPU_{n}.\]
 We shift it to obtain another fiber sequence
 \[BSU_{n}\rightarrow BPU_{n}\rightarrow K(\mathbb{Z}/n, 2),\]
 and consider the cohomological Serre spectral sequence with coefficients in $\mathbb{Q}$ and $\mathbb{Z}/m$. The first and second statement follows respectively from the vanishing of $H^{k}(K(\mathbb{Z}/n,2); \mathbb{Q})$ and $H^{k}(K(\mathbb{Z}/n,2); \mathbb{Z}/m)$, for $k>0$.
\end{proof}

\begin{remark}\label{torsion free}
 The map $\bar{P}$ factors as $BSU_{n}\rightarrow BU_{n}\xrightarrow{P} BPU_{n}$, where $P$ is induced by the quotient map $U_{n}\rightarrow PU_{n}$. By (1) of Lemma \ref{simplification},
 \[P^{*}:H^*(BPU_{n}; \mathbb{Z})\rightarrow H^{*}(BU_{n};\mathbb{Z})\]
 is a monomorphism modulo torsion elements. Furthermore, (1) of Lemma  \ref{simplification} shows that the torsion-free component of $H^{k}(BPU_{n};\mathbb{Z})$
 is $0$ if $k$ is odd, and is a finitely generated free abelian group of the same rank as the abelian group of homogenous polynomials in $\mathbb{Z}[c_{2}, c_{3}, \cdots, c_{n}]$ of degree $k/2$, if $k$ is even. Therefore the group structure of the torsion-free component of $H^{*}(BPU_{n};\mathbb{Z})$ is determined, though it doesn't  seem to have a canonical choice of generators. The calculation of the degree-wise torsion-free components is henceforth omitted unless we need some particular generators. However, we do check the cup products of the torsion-free elements. An alternative method to calculate the torsion-free component is via representation theory. Examples can be found in \cite{An2} (Antieau, 2016) and \cite{Ka2} (Kameko, 2016).
\end{remark}

We proceed to present a proof Theorem \ref{2p+2}.
\begin{lemma}\label{<p}
Let $p$ be a prime, $k, l\in\mathbb{N}$ satisfying $p|k$ and $l<p$. Then we have
\[p|{k\choose l}.\]
\end{lemma}
\begin{proof}
It follows from direct computation, or set $k_0=0$ in Lemma \ref{Lucas} (Lucas' Theorem).
\end{proof}
We present the following
\begin{proof}[Proof of Theorem \ref{2p+2}]
The case $p\nmid n$ follows from (2) of Lemma \ref{simplification}.

We assume $p|n$ for the rest of the proof. By Proposition \ref{K(Z,3)general}, the $p$-local cohomology of $K(\mathbb{Z},3)$ satisfies
\begin{equation*}
H^s(K(\mathbb{Z},3);\mathbb{Z}_{(p)})\cong
\begin{cases}
\mathbb{Z}_{(p)},\ s=3,\\
0,\ s\neq 3,s<2p+2,\\
\mathbb{Z}/p,\ s=2p+2.
\end{cases}
\end{equation*}
Therefore, the only possibly nontrivial differential landing in $^UE_{*}^{2p+2,0}$ is
\[^Ud_{2p-1}^{3,2p-2}: {^UE}_{2p-1}^{3,2p-2}\rightarrow {^UE}_{2p-1}^{2p+2,0}\cong{^UE}_{2}^{2p+2,0}\cong
H^{2p+2}(K(\mathbb{Z},3);\mathbb{Z}),\]
where
\[{^UE}_{2p-1}^{3,2p-2}\cong {^UE}_{2}^{3,2p-2}/\operatorname{Im}{^Ud}_{3}^{0,2p}=H^{2p-2}(BU_n;\mathbb{Z})\otimes_{\Z} x_1/\operatorname{Im}{^Ud}_{3}^{0,2p}.\]
Let $c_1^{i_1}\cdots c_{p-1}^{i_{p-1}}$ be a monomial in $H^{2p-2}(BU_n;\mathbb{Z})$. Then it follows from Corollary \ref{diff0} and Corollary \ref{cor2} that we have
\[^Ud_{2p-1}^{3,2p-2}(c_1^{i_1}\cdots c_{p-1}^{i_{p-1}}x_1)=\prod_{j=1}^{p-1}{n\choose j}^{i_j}y_{p,0},\]
where $y_{p,0}\in H^{2p+2}(K(\mathbb{Z},3);\mathbb{Z})$ is of order $p$. By Lemma \ref{<p}, we have
\begin{lemma}\label{2p-2}
If $p|n$, then we have.
\[^Ud_{2p-1}^{3,2p-2}=0.\]
\end{lemma}
Therefore we have $H^{2p+2}(BPU_n;\mathbb{Z})\cong\mathbb{Z}/p$, generated by $y_{p,0}$.

Finally, we verify the vanishing of the $p$-torsion subgroups of $H^k(BPU_n;\mathbb{Z})$ for $3<k<2p+2$. Notice that, the only values of $s$ and $t$ such that $s+t<2p+2$ and $^UE_2^{s,t}$ contain possibly nontrivial $p$-torsion elements are $s=3$ and $t=2q$, where $q<p$. For obvious degree reasons the only nontrivial differential landing in $^UE_*^{3,2q}$ is $^Ud_3^{0,2q+2}$. Therefore we have
\begin{equation}\label{eq:Im nabla}
^UE_{\infty}^{3,2q}\subset{^UE}_2^{3,2q}/\operatorname{Im}{^Ud}_3^{0,2q+2}\cong H^{2q}(BU_n;\mathbb{Z})/\nabla(H^{2q+2}(BU_n;\mathbb{Z})).
\end{equation}
We inspect
\[\nabla:H^{2q+2}(BU_n\mathbb{Z})\rightarrow H^{2q}(BU_n\mathbb{Z})\]
for $q<p$. By Corollary \ref{diff4} and Lemma \ref{Lucas} (Lucas' Theorem), in this range $\nabla$ is surjective after localization at $p$. In other words, we have
\begin{lemma}\label{lem:nablasurj}
When $q<p$, the group
\[^UE_2^{3,2p}/\operatorname{Im}{^Ud}_3^{0,2q+2}\cong H^{2q}(BU_n;\mathbb{Z})/\nabla(H^{2q+2}(BU_n;\mathbb{Z}))\]
is a torsion group containing no nontrivial $p$-torsion elements.
\end{lemma}
The desired conclusion now follows from (\ref{eq:Im nabla}) and Lemma \ref{lem:nablasurj}.
\end{proof}

We complete the proof of Theorem \ref{maincalc} in Sections \ref{1<k<6}, \ref{k=7,8} and \ref{k=9,10}. The reader may refer to Figure \ref{UE} for the spectral sequence ${{^U}E}_{*}^{*,*}$. For easier reference, we break the statement of the theorem down to the following list of assertions:
\begin{enumerate}[label=\textbf{C\arabic*}]
 \item \label{it:c1}$H^{k}(BPU_{n};\mathbb{Z})=0$ for $k=1, 2$. $H^{3}(BPU_{n};\mathbb{Z})\cong\mathbb{Z}/n$ is generated by $x_1$ of order $n$. $x_1^2$ is of order $2$ if $n$ is even and is $0$ otherwise.
 \item \label{it:c2}$H^{4}(BPU_{n};\mathbb{Z})\cong\mathbb{Z}$ is generated by $e_2$, such that $P^{*}(e_2)=2nc_2-(n-1)c_1^2$ when $n$ is even and $P^{*}(e_2)=nc_2-\frac{n-1}{2}c_1^2$ when $n$ is odd.
 \item \label{it:c3} $H^{5}(BPU_{n};\mathbb{Z})=0$.
 \item \label{it:c4} Let $n>2$. If $n$ is even, $H^{6}(BPU_{n};\mathbb{Z})\cong\mathbb{Z}\oplus\mathbb{Z}/2$ is generated by $e_3$ of infinite order and $x_1^2$ of order $2$. If $n$ is odd, $H^{6}(BPU_{n};\mathbb{Z})\cong\mathbb{Z}$ is generated by $e_3$, and $x_1^2=0$. In the exceptional case $n=2$, the assertion holds with the absence of $e_3$ and its corresponding direct summand $\mathbb{Z}$.
 \item \label{it:c5}If $n=4l+2$, $H^{7}(BPU_{n};\mathbb{Z})\cong\mathbb{Z}/2$ is generated by $e_{2}x_{1}$. Otherwise $H^{7}(BPU_{n};\mathbb{Z})=0$ and in particular $e_{2}x_{1}=0$.
 \item \label{it:c6}Let $n\geq 4$. If $3|n$, $H^{8}(BPU_{n};\mathbb{Z})\cong\mathbb{Z}\oplus\mathbb{Z}\oplus\mathbb{Z}/3$ generated by $e_4$ and $e_2^2$ of infinite order and $y_{3,0}$ of order $3$. Otherwise $H^{8}(BPU_{n};\mathbb{Z})\cong\mathbb{Z}\oplus\mathbb{Z}$ is generated by $e_2^2$ and $e_4$. In the exceptional cases $n=2, 3$, this assertion holds as well, with $e_4$ and its corresponding direct summand $\mathbb{Z}$ absent.
 \item \label{it:c7}The group $H^{9}(BPU_{n};\mathbb{Z})\cong\mathbb{Z}/\epsilon_{2}(n)$
 is generated by $x_1^3$ of order $\epsilon_{2}(n)$.
 \item \label{it:c8}$e_{3}x_1=0$.
 \item \label{it:c9}If $n\geq 5$, $H^{10}(BPU_{n};\mathbb{Z})\cong\mathbb{Z}\oplus\mathbb{Z}\oplus\mathbb{Z}/\epsilon_2(n)$
     is generated by $e_{2}e_{3}, e_5$ of infinite order, $y_{2,1}$ of order $\epsilon_2(n)$. In the exceptional cases $n<5$, the assertion holds, with the absence of monomials involving $e_i$ for $i>n$ and their corresponding direct summands $\mathbb{Z}$.
 \item \label{it:c10}$e_{2}x_1^2=y_{2,1}$ when $n=4l+2$, and $e_{2}x_1^2=0$ otherwise.
\end{enumerate}

\begin{figure}
\begin{tikzpicture}
\draw[step=1cm,gray, thick] (0,0) grid (8,8);
\draw [ thick, <->] (0,9)--(0,0)--(9,0);

\node [below] at (3,0) {$3$};
\node [below] at (6,0) {$6$};
\node [below] at (8,0) {$8$};
\node [below left] at (0,2) {$2$};
\node [below left] at (0,4) {$4$};
\node [below left] at (0,6) {$6$};
\node [below left] at (0,8) {$8$};
\node [right] at (9,0) {$s$};
\node [above] at (0,7) {$t$};

\node [above right] at (0,0) {$\mathbb{Z}$};
\node [above left] at (0,2) {$\mathbb{Z}:c_1$};
\node [above left] at (0,4) {$\mathbb{Z}^2:c_1^2,c_2$};
\node [above left] at (0,6) {$\mathbb{Z}^3:c_1^3,c_1c_2,c_3$};
\node [above left] at (0,8) {$\mathbb{Z}^5:c_1^4,c_1^2c_2,, c_2^2,c_1c_3,c_4$};

\node [above right] at (3,0) {$\mathbb{Z}$};
\node [above right] at (3,2) {$\mathbb{Z}$};
\node [above right] at (3,4) {$\mathbb{Z}^2$};

\node [above right] at (6,0) {$\mathbb{Z}/2$};
\node [above right] at (6,2) {$\mathbb{Z}/2$};
\node [above right] at (6,4) {$\mathbb{Z}/2\oplus\mathbb{Z}/2$};

\node [above right] at (8,0) {$\mathbb{Z}/3$};
\node [above right] at (8,2) {$\mathbb{Z}/3$};

\draw (0,2) [thick, ->] to node [above] {$\times n$} (3,0);
\draw (3,2) [thick, ->] to node [above] {$\times n$} (6,0);
\draw (3,4) [thick, ->] to (6,2);
\draw (3,6) [thick, ->] to (6,4);
\draw (0,4) [thick, ->] to (3,2);
\draw (0,6) [thick, ->] to (3,4);
\draw (0,8) [thick, ->] to (3,6);
\end{tikzpicture}\\
\caption{Some non-trivial differentials of $^{U}E_{*}^{*,*}$. A node with coordinate $(s,t)$ is unmarked if $^{U}E_{2}^{s,t}=0$.}\label{UE}
\end{figure}

\section{$H^{k}(BPU_n; \mathbb{Z})$ for $1\leq k\leq 6$}\label{1<k<6}
For $1\leq r\leq 5$ the results are given by Antieau and Williams (\cite{An}, 2014). The interested readers can compare it to our computation.
\begin{proof}[Proof of \ref{it:c1} to \ref{it:c4}]
 First notice ${{^U}E}^{s,t}_{2}=0$ for $s=1,2,4,5,7$ or $t$ odd. By Proposition \ref{diff3}, $^{U}d_{3}(c_{1})={^{T}}d_{3}(\sum_{i=1}^{n}v_{i})=nx_1$. This immediately proves \ref{it:c1}. See Figure 3 for a picture in low dimensions. By Corollary \ref{diff4} we have
 \begin{equation}\label{(0,4)}
 ^{U}d_{3}(c_2)=(n-1)c_{1}x_1, {^{U}d}_{3}(c_{1}^{2})=2nc_{1}x_1, {^{U}d}_{3}(c_{1}x_1)=nx_1^2.
 \end{equation}
 So $^{U}E_{4}^{0,4}\cong {^{U}E}_{\infty}^{0,4}\cong H^{4}(BPU_{n}; \mathbb{Z})$ is easily verified to be $\mathbb{Z}$. Interested readers can refer to Lemma 3.2 of \cite{An} (Antieau and Williams, 2014) to identify $H^{4}(BPU_{n}; \mathbb{Z})$ as a subgroup of $H^{4}(BU_{n};\mathbb{Z})$, or compute it directly. Either way, we proved \ref{it:c2}.

   Notice ${^{U}E}_{\infty}^{0,6}\cong\mathbb{Z}$, by Remark \ref{torsion free}.
 \begin{enumerate}
 \item If $n$ is even, by (\ref{(0,4)}), $^{U}d_{3}^{0,4}$ is a surjection and therefore $^{U}d_{3}^{3,2}=0$. Hence $H^{5}(BPU_{n}; \mathbb{Z})\cong {^{U}E}_{4}^{3,2}=0$, and
     \begin{equation*}
     H^{6}(BPU_{n}; \mathbb{Z})\cong {^{U}E}_{\infty}^{0,6}\oplus{^{U}E}_{4}^{6,0}\cong {^{U}E}_{\infty}^{0,6}\oplus{^{U}E}_{3}^{6,0}\cong
     \begin{cases}
     \mathbb{Z}\oplus\mathbb{Z}/2, n>2,\\
     \mathbb{Z}/2, n=2.
     \end{cases}
     \end{equation*}
     with ${^{U}E}_{3}^{6,0}$ generated by $x_1^2$.

 \item If $n$ is odd, by(\ref{(0,4)}) we have $^{U}d_{3}^{3,2}(c_{1}x_{1})=y_{2,0}$, which implies  ${^{U}E}_{4}^{6,0}=0$. Hence $H^{5}(BPU_{n}; \mathbb{Z})\cong {^{U}E}_{4}^{3,2}=0$, and $H^{6}(BPU_{n}; \mathbb{Z})\cong {^{U}E}_{\infty}^{0,6}\oplus{^{U}E}_{4}^{6,0}=\mathbb{Z}$.
 \end{enumerate}
 In both cases we take $e_3\in H^{6}(BPU_{n}; \mathbb{Z})$ generating ${^{U}E}_{\infty}^{0,6}$, for $n>2$. Therefore, \ref{it:c3} and \ref{it:c4} follows. Notice that $e_3$ is determined, up to torsion, by this argument. In fact, there is a unique choice of $e_3$ that fits the statement of Theorem \ref{maincalc}. This choice will be specified in Section 9, where we discuss the cup products in $H^{9}(BPU_n; \mathbb{Z})$.
 \end{proof}
\section{$H^{k}(BPU_n; \mathbb{Z})$ for $k=7,8$}\label{k=7,8}
 \begin{proof}[Proof of \ref{it:c5}]
 Notice that the only bi-degree $(s,t)$ such that $s+t=7$ and ${^U}E_{2}^{s,t}$ is nontrivial is $(3,4)$. Therefore $H^{7}(BPU_{n};\mathbb{Z})\cong {{^U}E}_{\infty}^{3,4}$.  We consider the $p$-local cohomology of ${H^{7}(BPU_{n};\mathbb{Z})}_{(p)}$ for each prime $p$ separately. The relevant differentials are as follows:
 \begin{enumerate}
  \item The only nontrivial differential landing in ${{^U}E}_{*}^{3,4}$ is ${{^U}d}_{3}^{0,6}: {{^U}E}_{3}^{0,6}\rightarrow {{^U}E}_{3}^{3,4}$
   \[c_3\mapsto (n-2)c_{2}x_{1},\quad c_{1}c_{2}\mapsto [nc_{2}+(n-1)c_{1}^2]x_1,\quad  c_{1}^3\mapsto 3nc_{1}^{2}x_{1}.\]
  \item The only nontrivial differential from ${{^U}E}_{*}^{3,4}$ is
  $${{^U}d}_{3}^{3,4}: {{^U}E}_{3}^{3,4}\rightarrow {{^U}E}_{3}^{6,2},$$
  $$c_{1}^{2}x_{1}\mapsto 2nc_{1}y_{2,0}=0,\quad c_{2}x_{1}\mapsto (n-1)c_{1}y_{2,0}.$$
 \end{enumerate}

 In what follows we compute the localization ${H^{7}(BPU_{n};\mathbb{Z})}_{(p)}$ for each prime $p$ separately. Throughout the rest of this and the next section we respect the following

 \begin{convention}
 Everything is implicitely assumed to be localized at the specified prime $p$ in each case.
 \end{convention}
 Case 1. $p=2.$ In this case ${{^U}E}_{\infty}^{3,4}=\operatorname{Ker}{{^U}d}_{3}^{3,4}/\textrm{Im}{{^U}d}_{3}^{0,6}$. By (2) of Lemma \ref{simplification} we only need to consider the case that $n$ is even, in which, assuming $n>2$ we have $\mathbb{Z}$-basis $\{c_{1}^{3}, c_{1}c_{2}, c_{3}\}$ and $\{c_{1}^{2}x_{1}, 2c_{2}x_{1}\}$ of ${{^U}E}_{3}^{0,6}$ and $\operatorname{Ker}{{^U}d}_{3}^{3,4}$, respectively. The corresponding matrix for ${{^U}d}_{3}^{0,6}$ is
 \begin{equation}\label{matrixeven}
\left[ \begin{array}{ccc}
 3n & n-1 & 0 \\
 0 & \frac{n}{2} & \frac{n-2}{2}
\end{array} \right].
\end{equation}

We apply invertible row and column operations to it and obtain
\begin{equation}\label{matrixeven2'}
\left[ \begin{array}{ccc}
 1 & 0 \\
 -\frac{n/2}{n-1} & 1
\end{array} \right]\cdot
\left[ \begin{array}{ccc}
 3n & n-1 & 0 \\
 0 & \frac{n}{2} & \frac{n-2}{2}
\end{array} \right]\cdot
\left[ \begin{array}{ccc}
 1 & 0 & 0 \\
 -\frac{3n}{n-1} & 1 & 0\\
 0 & 0 & 1
\end{array} \right]=
\left[ \begin{array}{ccc}
 0 & n-1 & 0 \\
 -\frac{3}{n-1}\cdot\frac{n^2}{2} & 0 & \frac{n-2}{2}
\end{array} \right],
\end{equation}
with the row operation corresponding to the change of basis
\begin{equation}\label{matrixeven3}
\left[ \begin{array}{ccc}
 c_{1}^{2}x_{1} & 2c_{2}x_{1} \\
\end{array} \right]\cdot
{\left[ \begin{array}{ccc}
 1 & 0 \\
 -\frac{n/2}{n-1} & 1
\end{array} \right]}^{-1}=
\left[ \begin{array}{ccc}
 c_{1}^{2}x_{1}+\frac{n}{n-1}c_{2}x_{1} & 2c_{2}x_{1} \\
\end{array} \right].
\end{equation}

By (\ref{matrixeven2'}) and (\ref{matrixeven3}), $\operatorname{Im}{{^U}d}_{3}^{0,6}$ is generated by \[c_{1}^{2}x_{1}+\frac{n/2}{n-1}c_{2}x_{1} \textrm{ and } \operatorname{gcd}\{\frac{n^2}{2},\frac{n-2}{2}\}\cdot 2c_{2}x_{1}=\operatorname{gcd}\{2,\frac{n-2}{2}\}\cdot 2c_{2}x_{1},\]
where
\begin{equation}\label{eq:2mod4}
\operatorname{gcd}\{2,\frac{n-2}{2}\}=
\begin{cases}
2,\ n\equiv2\pmod{4},\\
1,\textrm{ otherwise}.
\end{cases}
\end{equation}


In the exceptional case $n=2$, (\ref{matrixeven}) is reduced to
 \begin{equation*}
\left[ \begin{array}{ccc}
 6 & 1 \\
 0 & 1\\
\end{array} \right],
\end{equation*}
and (\ref{matrixeven2'}) is reduced to
\begin{equation}\label{matrixeven2}
\left[ \begin{array}{ccc}
 1 & 0 \\
 -\frac{n/2}{n-1} & 1
\end{array} \right]\cdot
\left[ \begin{array}{ccc}
 6 & 1  \\
 0 & 1
\end{array} \right]\cdot
\left[ \begin{array}{ccc}
 1 & 0 \\
 -\frac{3n}{n-1} & 1
\end{array} \right]=
\left[ \begin{array}{ccc}
 0 & 1 \\
 6 & 0
\end{array} \right],
\end{equation}
which yields the same argument as above.

The above discussion, in particular, (\ref{eq:2mod4}), shows
\begin{equation*}
{H^{7}(BPU_{n};\mathbb{Z})}_{(2)}\cong
\begin{cases}
\mathbb{Z}/2, \textrm{ generated by $2c_{2}x_1$, if } n\equiv2\pmod{4},\\
 0. \textrm{ otherwise}.
\end{cases}
\end{equation*}
By \ref{it:c2}, we know that for $n$ even, $H^{4}(BPU_{n};\mathbb{Z})\cong\mathbb{Z}$ is generated by $e_2$, which corresponds to $2nc_{2}-(n-1)c_{1}^{2}\in{{^U}E}_{\infty}^{0,4}$. Furthermore, when $n=4l+2$, in ${H^{7}(BPU_{n};\mathbb{Z})}_{(2)}\cong\mathbb{Z}/2$ we have
\begin{equation}\label{ringdim7}
 \begin{split}
  &2c_{2}x_{1}=3(4l+2)c_{2}x_{1}=3nc_{2}x_{1}\\
 =&[(n-1)c_{1}^{2}+nc_{2}]x_{1}-[(n-1)c_{1}^{2}-2nc_{2}]x_{1}\\
 =&[2nc_{2}-(n-1)c_{1}^{2}]x_{1}\in {{^U}E}_{4}^{3,4},
 \end{split}
\end{equation}
since by (\ref{matrixeven2}) and (\ref{matrixeven3}), $[(n-1)c_{1}^{2}+nc_{2}]x_{1}$ is in the image of ${{^U}d}_{3}^{0,6}$. Therefore, (\ref{ringdim7}) shows that $2c_{2}x_{1}=[2nc_{2}-(n-1)c_{1}^{2}]x_{1}$ yields a non-trivial product in ${{^U}E}_{\infty}^{3,4}$ that detects $e_{2}x_{1}$.

Case 2. $p>2$. By Lemma \ref{lem:nablasurj} the group
\[H^7(BPU_n;\mathbb{Z})\cong {^UE}_{\infty}^{3,4}\]
has no nontrivial $p$-torsion elements, from which \ref{it:c5} follows.
\end{proof}

\begin{proof}[Proof of \ref{it:c6}]
 It suffices to consider the torsion component. Therefore the relevant entries in ${{^U}E}_{2}^{*,*}$ are ${{^U}E}_{2}^{8,0}$ and ${{^U}E}_{2}^{6,2}$. The relevant differentials are ${{^U}d}_{3}^{3,4}$ and ${{^U}d}_{5}^{3,4}$, as given in (2) and (3) in the proof for $k=7$, together with the following one:
 \begin{equation}\label{(6,2) to (9,0)}
  {{^U}d}_{3}^{6,2}:{{^U}E}_{3}^{6,2}\rightarrow {{^U}E}_{3}^{9,0},\quad c_{1}y_{2,0}\mapsto nx_{1}y_{2,0}=nx_{1}^3.
 \end{equation}
 Again we consider ${H^{8}(BPU_{n};\mathbb{Z})}_{(p)}$ for each prime $p$ separately. In each of the following cases we assume that ${{^U}E}_{*}^{*,*}$ is localized at the specified prime $p$.

 Case 1. $p=2$. In this case the torsion component of ${H^{8}(BPU_{n};\mathbb{Z})}_{(2)}$ is isomorphic to ${{^U}E}_{\infty}^{6,2}=\operatorname{Ker}{{^U}d}_{3}^{6,2}/\operatorname{Im}{{^U}d}_{3}^{3,4}$. And by (2) of Lemma \ref{simplification} we consider only the case that $n$ is even. By (\ref{(6,2) to (9,0)}),  $\operatorname{Ker}{{^U}d}_{3}^{6,2}=\mathbb{Z}\{c_{1}y_{2,0}\}$. By the differentials given in (2) in the proof for \ref{it:c5}, $\operatorname{Im}{{^U}d}_{3}^{3,4}=\mathbb{Z}\{c_{1}y_{2,0}\}$. Therefore the torsion of ${H^{8}(BPU_{n};\mathbb{Z})}_{(2)}$ is $0$ when $n$ is even.

 Case 2. $p=3$. In this case the torsion component of ${H^{8}(BPU_{n};\mathbb{Z})}_{(3)}$ is isomorphic to ${{^U}E}_{2}^{8,0}/\operatorname{Im}{{^U}d}_{5}^{3,4}$, where ${{^U}E}_{2}^{8,0}\cong\mathbb{Z}/3$ is generated by a single element $y_{3,0}$. By the differentials given in (3) of the proof for $k=7$, $\operatorname{Im}{{^U}d}_{5}^{3,4}=0$ when $3|n$. Therefore the torsion component of ${H^{8}(BPU_{n};\mathbb{Z})}_{(3)}$ is isomorphic to $\mathbb{Z}/3$ and generated by $y_{3,0}$.

 For $p>3$ there is no $p$-torsion in the relevant range of ${{^U}E}_{*}^{*,*}$.

 It remains to show  that $e_2^2$ generates a $\mathbb{Z}$-summand. In the proof of \ref{it:c5} we have shown that $e_2$ is represented by the permanent cocycle
 \begin{equation}\label{e2}
 \epsilon_2(n-1)^{-1}[2nc_{2}-(n-1)c_{1}^{2}]\in{^UE}_{\infty}^{0,4}.
 \end{equation}
 Therefore $e_2^2$ is represented by the permanent cocycle
 \[\epsilon_2(n-1)^{-2}[4n^2c_2^2-4(n-1)nc_{1}^{2}c_2+(n-1)^2c_1^4]\in{^UE}_{\infty}^{0,8}.\]
 Comparing the coefficients, we see that it cannot be divided by any integer other than $ \pm1$. It follows from the discussion in Remark \ref{torsion free} that $e_2^2$ is a generator.
\end{proof}

\section{$H^{k}(BPU_n; \mathbb{Z})$ for $k=9, 10$}\label{k=9,10}

The study of $H^{9}(BPU_n; \mathbb{Z})$ requires extra work when $n$ is even, since the differential ${^{U}}d_{9}^{0,8}$ cannot be determined from Proposition \ref{main}. Indeed, when $n$ is even, the differential
\[{^{U}}d_{9}^{0,8}: {^{U}}E_{9}^{0,8}\rightarrow {^{U}}E_{9}^{9,0}\]
does not satisfy the condition of Proposition \ref{main}, since the differential
\[{^{T}}d_{3}^{6,2}: {^{T}}E_{3}^{6,2}\rightarrow {^{T}}E_{3}^{9,0}\]
is not trivial. Indeed, we have
\[{^{T}}d_{3}^{6,2}(v_1x_1^2)=x_1^3\neq 0\in {^TE}_3^{9,0}.\]
On the other hand, we use Corollary \ref{diff4} to compute
\[{^Ud}_{3}^{6,2}(c_1x_1^2)=nx_1^3=0,\]
from which we deduce ${^Ud}_{3}^{6,2}=0$. Therefore, the homomorphism induced by the inclusion of maximal torus $\Psi^*:{^{U}}E_{4}^{0,8}\rightarrow{^{T}}E_{4}^{0,8}$ is not injective, and Proposition \ref{main} fails to determine ${^U}d_{9}^{0,8}$.

To remedy this, we consider the exceptional isomorphism of Lie groups $PU_2\cong SO_3$. It is well-known that we have
\begin{equation}\label{BSO_3mod2}
H^*(BPU_2;\mathbb{Z}/2)\cong H^*(BSO_3;\mathbb{Z}/2)\cong\mathbb{Z}/2[\omega_2,\omega_3],
\end{equation}
where $\omega_2, \omega_3$ are the universal Stiefel-Whitney classes. Furthermore, we have
\begin{equation}\label{Sq2w3}
\operatorname{Sq}^2(\omega_3)=\omega_2\omega_3.
\end{equation}
Indeed, the Adem relations indicates
\[\operatorname{Sq}^1\operatorname{Sq}^2(\omega_3)=\operatorname{Sq}^3(\omega_3)=\omega_3^2\neq 0.\]
Therefore $\operatorname{Sq}^2(\omega_3)\neq 0$. The only nonzero element in $H^5(BPU_2;\mathbb{Z}/2)$ is $\omega_2\omega_3$, and equation (\ref{Sq2w3}) follows.

As for the integral cohomology, we have the following result due to E. Brown. For more general cases see \cite{Br} (Brown, 1982).

\begin{proposition}\label{BSO_3}
 $H^{*}(BPU_2;\mathbb{Z})\cong H^{*}(BSO_3;\mathbb{Z})\cong\mathbb{Z}[e_2]\otimes_{\mathbb{Z}}\mathbb{Z}/2[x_1]$, where $e_2$ is of degree $4$, and we abuse the notation $x_1$ to let it denote the image of $x_1\in H^{3}(K(\mathbb{Z},3);\mathbb{Z})$ under the homomorphism induced by the second arrow of the fiber sequence $BU_2\rightarrow BPU_2\rightarrow K(\mathbb{Z},3)$.
\end{proposition}
To relate $BPU_2$ and $BPU_n$ for $n$ even, we consider the map $\Delta: BU_2\rightarrow BU_n$ be the one given by the inclusion
\begin{equation*}
U_2\hookrightarrow U_n,\ A\mapsto
\begin{bmatrix}
    A & \hdots & 0 & 0 \\
    0 & A &\hdots &\vdots\\
    \vdots& \cdots& \ddots   &0\\
    0& \cdots& \hdots & A
    \end{bmatrix},
\end{equation*}
 This inclusion passes to $PU_2\hookrightarrow PU_n$ and yields a map $\Delta':BPU_2\hookrightarrow BPU_n$.

Recall that in Theorem \ref{maincalc} we let $e_2$ denote the generator of $H^{4}(BPU_n;\mathbb{Z})\cong\mathbb{Z}$ for all $n>1$.
Part (1) of the following lemma is due to A. Bousfield.
\begin{lemma}\label{k=9lemma1}
 For $n>0$ even, we have the following assertions.
 \begin{enumerate}
 \item $(\Delta')^{*}(x_1)=x_1$. In particular, $x_1^3\in H^{9}(BPU_{n};\mathbb{Z})$ is non-zero.
 \item $(\Delta')^{*}(e_2)=(\frac{n}{2})^{2}e_2$. In particular, when $n=4l+2$ for some integer $l$, $e_2x_1^2\in H^{10}(BPU_{n};\mathbb{Z})$ is non-zero.
 \end{enumerate}
\end{lemma}

\begin{proof}
 Consider the following commutative diagram:
\begin{equation*}
\begin{tikzcd}
BS^1\arrow{d}{=}\arrow{r}&BU_2 \arrow{d}{\Delta}\arrow{r} &BPU_{2} \arrow{d}{\Delta'}\\
BS^1\arrow{r}            &BU_{n}\arrow{r}                &BPU_{n}
\end{tikzcd}
\end{equation*}
which induces another commutative diagram as follows:
\begin{equation}\label{Delta}
\begin{tikzcd}
BU_2 \arrow{d}{\Delta}\arrow{r} &BPU_{2} \arrow{d}{\Delta'}\arrow{r} &K(\mathbb{Z},3)\arrow{d}{=}\\
BU_{n}\arrow{r}               &BPU_{n} \arrow{r}                   &K(\mathbb{Z},3)
\end{tikzcd}
\end{equation}

 This diagram induces a homomorphism of Serre spectral sequences, such that its restriction on the bottom row of the $E_{2}$ pages is the identity. In particular it takes $x_1$ to itself. Moreover, it follows from Proposition \ref{BSO_3} that $(\Delta')^{*}(x_1^3)=x_1^3\in H^{9}(BPU_2;\mathbb{Z})$ is non-zero, which completes the proof of (1).

 It is well known that the $U_{n}$-bundle over $BU_2$ induced by $\Delta$ is the Whitney sum of $\frac{n}{2}$ copies of the universal $U_2$-bundle, of which the total Chern class is
 \[(1+c_1+c_2)^{\frac{n}{2}}=1+\frac{n}{2}c_1+(\frac{n}{2}c_{2}+{{n/2}\choose{2}}c_1^2)+(\textrm{terms of higher degrees}).\]
 Therefore we have
 \begin{equation}\label{c_1}
 \begin{split}
 &\Delta^{*}: H^{2}(BU_n;\mathbb{Z})\rightarrow H^{2}(BU_2;\mathbb{Z}),\\
 & c_1\mapsto \frac{n}{2}c_1,\quad c_2\mapsto \frac{n}{2}c_{2}+{{n/2}\choose{2}}c_1^2=\frac{n}{2}c_{2}+\frac{n(n-2)}{8}c_1^2.
 \end{split}
 \end{equation}
 In particular,
 $$\Delta^{*}((n-1)c_1^2-2nc_2)=(n-1)(\frac{n}{2})^{2}c_1^2-2n[\frac{n}{2}c_{2}+\frac{n(n-2)}{8}c_1^2]
 =(\frac{n}{2})^{2}(c_1^2-4c_2).$$
 Notice that when $n=2$, $P^{*}(e_2)=c_1^2-4c_2$, and the equation above implies $(\Delta')^{*}(e_2)=(\frac{n}{2})^{2}e$. When $n=4l+2$, we have $(\Delta')^{*}(e_{2}x_1^2)=(2l+1)^{2}e_{2}x_1^2=e_{2}x_1^2\in H^{10}(BPU_{2};\mathbb{Z})$ which is non-zero by Proposition \ref{BSO_3}, and (2) follows.
\end{proof}

\begin{proof}[Proof of \ref{it:c7}, \ref{it:c8}]
 The relevant entries in ${{^U}E}_{2}^{*,*}$ are ${{^U}E}_{2}^{3,6}$ and ${{^U}E}_{2}^{9,0}$, the latter being isomorphic to $H^9(K(\mathbb{Z},3);\mathbb{Z})\cong\mathbb{Z}/2$ and generated by $x_1^3$. We study the localization of ${{^U}E}_{*}^{*,*}$ at each prime $p$ separately.

 Case 1. $p=2$. Again we consider only the case that $n$ is even. When $n>2$, with respect to the basis $\{c_{1}^{3}x_{1}, c_{1}c_{2}x_{1}, c_{3}x_{1}\}$ for ${{^U}E}_{2}^{3,6}$ and $\{c_{1}^{2}y_{2,0}, c_{2}y_{2,0}\}$ for ${{^U}E}_{2}^{6,4}$, the differential ${{^U}d}_{3}^{3,6}: {{^U}E}_{2}^{3,6}\rightarrow {{^U}E}_{2}^{6,4}$ is represented by the following matrix:
 \begin{equation}\label{k=9 p=2 matrix1}
  \left[ \begin{array}{ccc}
  3n & n-1 & 0 \\
  0 & n & n-2
  \end{array} \right].
 \end{equation}
 Since $y_{2,0}$ is of order $2$, and $n$ is even, (\ref{k=9 p=2 matrix1}) implies that $\operatorname{Ker}{{^U}d}_{3}^{3,6}$ has a basis $\{c_{1}^{3}x_{1}, 2c_{1}c_{2}x_{1}, c_{3}x_{1}\}$. Taking $\{y_{2,1}\}$ as a basis for ${{^U}E}_{2}^{10,0}$, ${{^U}d}_{7}^{3,6}$ is represented by the matrix

 \begin{equation}\label{k=9 p=2 matrix2}
  \left[ \begin{array}{ccc}
  \frac{n^3}{2} & \frac{n^{2}(n-1)}{2} & \frac{n(n-1)(n-2)}{12}
  \end{array} \right].
 \end{equation}
All three entries being even, we have $\operatorname{Ker}{{^U}d}_{3}^{3,6}=\operatorname{Ker}{{^U}d}_{7}^{3,6}=\mathbb\{c_{1}^{3}x_{1}, 2c_{1}c_{2}x_{1}, c_{3}x_{1}\}$.
 With this basis and $\{c_{1}^{4}, c_{1}^{2}c_{2}, c_{1}c_{3}, c_{2}^{2}, c_{4}\}$ as a basis for ${{^U}E}_{2}^{0,8}$, ${{^U}d}_{3}^{0,8}: {{^U}E}_{2}^{0,8}\rightarrow \operatorname{Ker}{{^U}d}_{7}^{3,6}\subset {{^U}E}_{7}^{3,6}$ is represented by the following matrix:

 \begin{equation}\label{k=9 p=2 matrix3}
  \left[ \begin{array}{ccccc}
  4n & n-1 & 0 & 0 & 0 \\
  0 & n & \frac{n-2}{2} & n-1 & 0 \\
  0 & 0 & n & 0 & n-3
 \end{array} \right]
 \end{equation}
 We apply an invertible column operation on it as follows:
 \begin{equation}\label{k=9 p=2 matrix4}
 \begin{split}
 &  \left[ \begin{array}{ccccc}
  4n & n-1 & 0 & 0 & 0 \\
  0 & n & \frac{n-2}{2} & n-1 & 0 \\
  0 & 0 & n & 0 & n-3
 \end{array} \right]\cdot
  \left[ \begin{array}{ccccc}
  1 & 0               & 0                   & 0 & 0 \\
  0 & 1               & 0                   & 0 & 0\\
  0 & 0               & 1                   & 0 & 0\\
  0 & -\frac{n}{n-1} & -\frac{n-2}{2(n-1)} & 1 & 0\\
  0 & 0               & 0                   & 0 & 1
 \end{array} \right]\\ &=
 \left[ \begin{array}{ccccc}
  4n & n-1 & 0 & 0 & 0 \\
  0 & 0 & 0 & n-1 & 0 \\
  0 & 0 & n & 0 & n-3
 \end{array} \right].
 \end{split}
 \end{equation}
 Therefore ${{^U}d}_{3}^{0,8}: {{^U}E}_{2}^{0,8}\rightarrow \operatorname{Ker}{{^U}d}_{7}^{3,6}$ is onto and ${{^U}E}_{\infty}^{3,6}=0$.

 In the exceptional case $n=2$, we have $c_{3}, c_{4}=0$, and (\ref{k=9 p=2 matrix2}), (\ref{k=9 p=2 matrix3}), are respectively reduced to

 \begin{equation*}
  \left[ \begin{array}{cc}
  4 & 2
  \end{array} \right],
  \left[ \begin{array}{ccc}
  8 & 1 & 0 \\
  0 & 2 & 1
 \end{array} \right].
 \end{equation*}
 Consequently, ${{^U}d}_{3}^{0,8}: {{^U}E}_{2}^{0,8}\rightarrow \operatorname{Ker}{{^U}d}_{7}^{3,6}$ is onto as well.
 It remains to consider
 $${{^U}d}_{9}^{0,8}: {{^U}E}_{9}^{0,8}\rightarrow {{^U}E}_{9}^{9,0}\cong {{^U}E}_{2}^{9,0}=\mathbb{Z}/2\{x_{1}^{3}\},$$
 Since ${{^U}E}_{9}^{9,0}\cong\mathbb{Z}/2$ is generated by $x_{1}^{3}$, ${{^U}d}_{9}^{0,8}$ is either surjective or $0$, depending on whether $x_{1}^{3}$ is $0$ or not. By Lemma \ref{k=9lemma1}, $x_{1}^3\neq 0$ when $n$ is even. Therefore ${{^U}d}_{9}^{0,8}=0$.

 Summarising all above, when $n$ is even, ${H^{9}(BPU_{n};\mathbb{Z})}_{(2)}=\mathbb{Z}/2$ is generated by $x_{1}^3$, which proves \ref{it:c7}.  We proceed to study the cup products in ${H^{9}(BPU_{n};\mathbb{Z})}_{(2)}$. Since ${{^U}E}_{\infty}^{3,6}=0$, $x_{1}e_3=x_{1}^3$ or $x_{1}e_3=0$. This merely depends on a choice of $e_3$: if the former case is true, then we simply replace $e_3$ by $e_3+x_1^2$ to obtain the latter case. Therefore the $2$-local case of \ref{it:c8} follows.

 Case 2. $p=3$. Since $^UE_2^{6,4}$ is $2$-torsion, we have
 \begin{equation*}
 {^UE}_{5}^{3,6}={^UE}_{3}^{3,6}/\operatorname{Im}{^Ud}_{3}^{0,8}.
 \end{equation*}

 When $n>3$, we take basis $\{c_{1}^{4}, c_{1}^{2}c_{2}, c_{1}c_{3}, c_{2}^{2}, c_{4}\}$ for ${{^U}E}_{3}^{0,8}$ and $\{c_{1}^{3}x_{1}, c_{1}c_{2}x_{1}, c_{3}x_{1}\}$ for ${{^U}E}_{2}^{3,6}$. Then the matrix representing ${{^U}d}_{3}^{0,8}$ is the following:
 \begin{equation}\label{k=9 p=3 matrix1}
  \left[ \begin{array}{ccccc}
  4n & n-1 & 0   & 0      & 0 \\
  0  & 2n  & n-2 & 2(n-1) & 0 \\
  0  & 0   & n   & 0      & n-3
 \end{array} \right].
 \end{equation}

 We consider only the case $3|n$, in which we apply an invertible column operation to \ref{k=9 p=3 matrix1} and obtain
 \begin{equation}\label{k=9 p=3 matrix2}
 \begin{split}
 & \left[ \begin{array}{ccccc}
  4n & n-1 & 0   & 0      & 0 \\
  0  & 2n  & n-2 & 2(n-1) & 0 \\
  0  & 0   & n   & 0      & n-3
 \end{array} \right]\cdot
 \left[ \begin{array}{ccccc}
  1 & 0              & 0                   & 0 & 0 \\
  0 & 1              & 0                   & 0 & 0 \\
  0 & 0              & 1                   & 0 & 0 \\
  0 & -\frac{n}{n-1} & -\frac{n-2}{2(n-1)} & 1 & 0\\
  0 & 0              & 0                   & 0 & 1
 \end{array} \right]\\=&
 \left[ \begin{array}{ccccc}
  4n & n-1 & 0 & 0      & 0 \\
  0  & 0   & 0 & 2(n-1) & 0 \\
  0  & 0   & n & 0      & n-3
 \end{array} \right],
 \end{split}
 \end{equation}
In the exceptional case $n=3$, the vanishing of $c_4$ reduces the matrices (\ref{k=9 p=3 matrix1}) and (\ref{k=9 p=3 matrix2}) to
 \begin{equation*}
  \left[ \begin{array}{ccccc}
  4n & n-1 & 0   & 0       \\
  0  & 2n  & n-2 & 2(n-1)  \\
  0  & 0   & n   & 0
 \end{array} \right]\quad\textrm{and}\quad
 \left[ \begin{array}{ccccc}
  4n & n-1 & 0 & 0       \\
  0  & 0   & 0 & 2(n-1)  \\
  0  & 0   & n & 0
 \end{array} \right],
 \end{equation*}
 respectively.

 Summarizing the above, given suitable choices of basis we may express $^Ud_3^{0,8}$ in terms of matrices
 as follows:
 \begin{equation*}
 \begin{cases}
  {{^U}d}_{3}^{0,8}=
\begin{bmatrix}
  4n & n-1 & 0 & 0      & 0 \\
  0  & 0   & 0 & 2(n-1) & 0 \\
  0  & 0   & n & 0      & n-3
 \end{bmatrix},\
 n>3,\\
  {{^U}d}_{3}^{0,8}=
 \begin{bmatrix}
  4n & n-1 & 0 & 0       \\
  0  & 0   & 0 & 2(n-1)  \\
  0  & 0   & n & 0
 \end{bmatrix},\
 n=3.
 \end{cases}
 \end{equation*}
In both cases it is elementary to verify
 \begin{equation}\label{UE53,6}
  {^UE}_{5}^{3,6}\cong{^UE}_{2}^{3,6}/\operatorname{Im}{{^U}d}_{3}^{0,8}\cong\mathbb{Z}/3 \textrm{ generated by } c_{3}x_{1}.
 \end{equation}
 The only nontrivial target of any differential with domain ${{^U}E}_{2}^{3,6}$ is ${{^U}E}_{2}^{8,2}$, generated by $c_1y_{3,0}$, a $3$-torsion element. The differential ${{^U}E}_{5}^{3,6}\xrightarrow{{^Ud}_{5}^{3,6}} {^UE}_{5}^{8,2}$
 is given by the following lemma, of which the proof we postpone until after this one.
\begin{lemma}\label{lem:d_5^(3,6)}
Assuming $3|n$, the differential ${^Ud}_{5}^{3,6}:{{^U}E}_{5}^{3,6}\rightarrow{^UE}_{5}^{8,2}$ is an isomorphism.
\end{lemma}
Therefore, in the $3$-local case we have
 \[{^UE}_{\infty}^{3,6}\cong{^UE}_{6}^{3,6}=0,\]
 and
 \[H^9(BPU_n;\mathbb{Z})_{(3)}=0.\]
 This proves the $3$-local case of \ref{it:c8}.

 Case 3. $p>3$. By Theorem \ref{2p+2} there is no nontrivial torsion class in this case, and we conclude the proof of \ref{it:c7} and \ref{it:c8}.
\end{proof}

 \begin{proof}[Proof of Lemma \ref{lem:d_5^(3,6)}] Since everything in sight is $3$-local, for degree reasons, there is no nontrivial differential ${^TE}_{r}^{*,*}$ , $r<5$, landing in ${^TE}_{*}^{8,2}$. It follows from Proposition \ref{main} that we have the following commutative diagram:
 \begin{equation}\label{Td3,6}
 \begin{tikzcd}
 {^UE}_{5}^{3,6}\arrow{r}{{^Ud}_{5}^{3,6}}\arrow{d}{\Psi^*}&{^UE}_{5}^{8,2}\arrow{d}{\Psi^*}\\
 {^TE}_{5}^{3,6}\arrow{r}{^Td_{5}^{3,6}}&{^TE}_{5}^{8,2}
  \end{tikzcd}
 \end{equation}
 where the vertical arrows $\Psi^*$ are induced by the inclusion of maximal tori, sending a universal Chern $c_i$ to the $i$th elementary symmetric polynomial in $\{v_i\}_{1\leq i\leq n}$, and restricts to the identity on $H^*(K(\mathbb{Z},3);\mathbb{Z})$.

 On the other hand, by (\ref{UE53,6}) it suffices to consider ${^Ud}_{5}^{3,6}(c_3x_1)$. To apply diagram (\ref{Td3,6}), we recall (Proposition \ref{diff2}) that in ${^TE}_{*}^{*,*}$ we have permanent cocycles
 $\{v_i'=v_i-v_n\in {^TE}_{\infty}^{0,2}|1\leq i\leq n\}$.
 Therefore, we have
 \begin{equation}\label{c3x1}
 \begin{split}
  &\Psi^*{^Ud}_{5}^{3,6}(c_3x_1)={^Td}_{5}^{3,6}\Psi^*(c_3x_1)\\
 =&{^Td}_{5}^{3,6}(\sum_{1\leq i<j<k\leq n}v_iv_jv_k x_1)\\
 =&{^Td}_{5}^{3,6}[\sum_{1\leq i<j<k\leq n}(v_i'+v_n)(v_j'+v_n)(v_k'+v_n)x_1]\\
 =&{^Td}_{5}^{3,6}\{\sum_{1\leq i<j<k\leq n}[v_i'v_j'v_k'x_1+(v_i'v_j'+v_j'v_k'+v_k'v_i')v_nx_1+\\
 &(v_i'+v_j'+v_k')v_n^2x_1+v_n^3x_1]\}.
 \end{split}
 \end{equation}
 Notice that $v_i'v_j'v_k' x_1$ is a permanent cocycle since all of its factors are. The same holds for $(v_i'v_j'+v_j'v_k'+v_k'v_i')v_nx_1$: Indeed, in the spectral sequence $^KE_*^{*,*}$, we consider the element $vx_1\in{^KE}_2^{3,6}$, where $v\in {^KE}_2^{0,2}$ is the fundamental class of $H^2(K(\mathbb{Z},2);\mathbb{Z})$.
 We have
 \[^Kd_3^{3,6}(vx_1)=x^2_1\in{^KE}_3^{6,0},\]
 a $2$-torsion, which means $2vx_1\in {^KE}_{4}^{3,2}$, and moreover, for obvious degree reasons, it is a permanent cocycle. Then it follows from Proposition \ref{diff1} that the same is true for $v_nx_1$ in ${^TE}_2^{3,6}$. Furthermore, in the spectral sequence $^KE_*^{*,*}$ we have
 \[^Kd_5^{3,6}(v^3x_1)=0,\]
 since in $^KE_5^{3,6}$, the element $v^3x_1$ is $4$-torsion whereas the target of $^Kd_5^{3,6}$ is isomorphic to $\mathbb{Z}/3$. By Proposition \ref{diff1}, the same holds for $v_n^3x_1$. Summarizing the above, we have
 \begin{equation}\label{vanishingd5}
 {^Td}_{5}^{3,6}(v_i'v_j'v_k' x_1)={^Td}_{5}^{3,6}((v_i'v_j'+v_j'v_k'+v_k'v_i')v_nx_1)={^Td}_{5}^{3,6}(v_n^3x_1)=0.
 \end{equation}
 On the other hand, recall from Corollary \ref{diff0} that in $^KE_*^{*,*}$ we have $^Kd_{5}^{3,4}(v^2x_1)=y_{3,0}$. By Proposition \ref{diff1}, we have
 \begin{equation}\label{vn^2x1}
 ^Td_{5}^{3,4}(v_n^2x_1)=y_{3,0}.
 \end{equation}
 We resume the computation (\ref{c3x1}):
 \begin{equation}\label{c3x1cont}
 \begin{split}
 &\Psi^*{^Ud}_{5}^{3,6}(c_3x_1)\\
 =&{^Td}_{5}^{3,6}\{\sum_{1\leq i<j<k\leq n}[v_i'v_j'v_k'x_1+(v_i'v_j'+v_j'v_k'+v_k'v_i')v_nx_1+\\
 &(v_i'+v_j'+v_k')v_n^2x_1+v_n^3x_1]\} \\
 =&{^Td}_{5}^{3,6}[\sum_{1\leq i<j<k\leq n}(v_i'+v_j'+v_k')v_n^2x_1]\textrm{ (because of (\ref{vanishingd5}))}\\
 =&\sum_{1\leq i<j<k\leq n}(v_i'+v_j'+v_k'){^Td}_{5}^{3,4}(v_n^2x_1)\\
 =&\sum_{1\leq i<j<k\leq n}(v_i'+v_j'+v_k')y_{3,0} \textrm{ (because of (\ref{vn^2x1}))}\\
 =&\sum_{1\leq i<j<k\leq n}(v_i+v_j+v_k)y_{3,0}\textrm{ (since $y_{3,0}$ is $3$-torsion)},
 \end{split}
 \end{equation}
 which is symmetric in the variables $v_i$, $1\leq i\leq n$. For degree reasons it has to be $mc_1y_{3,0}$ for some integer $m$. To determine $m$, consider the chain of maps
 \[K(\Z,2)=BS^1\rightarrow BT^n\rightarrow BU_n\]
 induced by the inclusions of Lie groups. The restriction on cohomology of the composition takes $mc_1$ to $mnv$. On the other hand, the restriction of $K(\Z,2)=BS^1\rightarrow BT^n$ takes $\sum_{1\leq i<j<k\leq n}(v_i+v_j+v_k)$ to $3{n\choose 3}v$. A comparison then shows $mn=3{n\choose 3}$, or equivalently
 \begin{equation}\label{mc1}
 m=3{n\choose 3}n^{-1}=\frac{(n-1)(n-2)}{2}.
 \end{equation}
 By (\ref{mc1}) and the injectivity of ${^UE}_{5}^{8,2}\xrightarrow{\Psi^*} {^TE}_{5}^{8,2}$ we have
 \begin{equation}\label{c3x1final}
 {^Ud}_{5}^{3,6}(c_3x_1)=\frac{(n-1)(n-2)}{2}c_1y_{3,0}=c_1y_{3,0},
 \end{equation}
 since $y_{3,0}$ is $3$-torsion and we assume $3|n$, and the lemma follows.
 \end{proof}

\begin{proof}[Proof of \ref{it:c9}, \ref{it:c10}]
  Case 1. $p=2$. In this case the relevant entries are ${{^U}E}_{2}^{6,4}$ and ${{^U}E}_{2}^{10,0}$. First suppose that $n>2$. Fix a basis $\{c_{1}^{2}y_{2,0}, c_{2}y_{2,0}\}$ for ${{^U}E}_{2}^{6,4}$ and $\{c_{1}x_{1}y_{2,0}\}=\{c_{1}x_{1}^3\}$ for ${{^U}E}_{2}^{9,2}$, and ${{^U}d}_{3}^{6,4}: {{^U}E}_{2}^{6,4}\rightarrow {{^U}E}_{2}^{9,2}$ is represented by the matrix
 \begin{equation}
 \left[\begin{array}{cc}
 2n & n-1
 \end{array}\right],
 \end{equation}
 which shows that $\operatorname{Ker}{{^U}d}_{3}^{6,4}$ is generated by $\{c_{1}^{2}y_{2,0}\}$. On the other hand, by (\ref{k=9 p=2 matrix1}), $\operatorname{Im}{{^U}d}_{3}^{3,6}$ is also generated by $\{c_{1}^{2}y_{2,0}\}$. Therefore
 \begin{equation}\label{k=10(6,4)}
 {{^U}E}_{\infty}^{6,4}\cong {{^U}E}_{5}^{6,4}\cong\operatorname{Ker}{{^U}d}_{3}^{6,4}/\operatorname{Im}{{^U}d}_{3}^{3,6}=0.
 \end{equation}
 By (\ref{k=9 p=2 matrix2}) ${{^U}d}_{7}^{3,6}=0$. Hence
 \begin{equation}\label{k=10(10,0)}
 {{^U}E}_{\infty}^{10,0}\cong {{^U}E}_{2}^{10,0}/\operatorname{Im}{{^U}d}_{7}^{3,6}\cong {{^U}E}_{2}^{10,0}\cong\mathbb{Z}/2.
 \end{equation}
 is generated by $x_{1}y_{2,0}$. In the exceptional case $n=2$, $c_{3}=0$ and all the arguments above hold as well, since all the differentials of $c_{3}x_{1}$ are $0$ anyway.

 Therefore, the torsion component of ${H^{10}(BPU_{n}; \mathbb{Z})}_{(2)}$ is isomorphic to $\mathbb{Z}/2$, generated by $y_{2,1}$ when $n$ is even.

 Case 2. $p=3$. We take $3|n$. In this case the only relevant entry is ${{^U}E}_{2}^{8,2}$. By previous calculation (\ref{c3x1final}) we see that ${^Ud}_{5}^{3,6}$ is onto ${{^U}E}_{2}^{8,2}$. Therefore we have ${{^U}E}_{\infty}^{8,2}=0$, and $H^{10}(BPU_{n}; \mathbb{Z})$ is $3$-torsion-free.

 Case 3. $p>3$. By Theorem \ref{2p+2} there is no nontrivial $p$-torsion element in degree $10$.

 It remains to show that $e_2e_3$ generates a $\mathbb{Z}$-summand when $n\geq3$. A close look at the differential ${^Ud}_3^{0,6}$ shows that $e_3$ is represented by the permanent cocycle
 \[\epsilon_2(n)^{-1}\epsilon_3(n-1)^{-1}\epsilon_3(n-2)^{-1}[3n^2c_3-3n(n-2)c_1c_2-(n-1)(n-2)c_1^3]
 \in{^UE}_{\infty}^{0,6}.\]
 Then it follows from (\ref{e2}) that $e_2e_3$ is represented by the permanent cocycle
 \begin{equation*}
 \begin{split}
 &\epsilon_2(n-1)^{-1}[2nc_{2}-(n-1)c_{1}^{2}]\cdot\\
 &\epsilon_2(n)^{-1}\epsilon_3(n-1)^{-1}\epsilon_3(n-2)^{-1}[3n^2c_3-3n(n-2)c_1c_2-(n-1)(n-2)c_1^3]\\
 =&\epsilon_2(n-1)^{-1}\epsilon_2(n)^{-1}\epsilon_3(n-1)^{-1}\epsilon_3(n-2)^{-1}\cdot\\
 &[6n^3c_2c_3-3n^2(n-1)c_1^2c_3+(n-2)(n-1)nc_1^3c_2-6(n-2)n^2c_1c_2^2+\\
 &(n-2)(n-1)^2c_1^5.]
 \end{split}
 \end{equation*}
 Comparing the coefficients, we see that it cannot be divided by any integer other than $\pm1$. Therefore $e_2e_3$ is a generator.

 The assertion \ref{it:c9} follows from the discussion above.

 We proceed to study cup products in $H^{10}(BPU_{n}; \mathbb{Z})$. Recall that when $n$ is even,  $H^{4}(BPU_{n}; \mathbb{Z})\cong\mathbb{Z}$ is generated by an element $e_2$ such that $P^{*}(e_2)=2nc_2-(n-1)c_{1}^{2}\in H^{4}(BU_{n}; \mathbb{Z})$. By (\ref{k=10(6,4)}), $e_{2}x_{1}^2\in{{^U}E}_{\infty}^{6,4}$ is a coboundary.  Therefore, we have either $e_2x_1^2=0$,  or $e_2x_{1}^2=y_{2,0}$. The former is true when $4|n$, because in this case we have $e_2x_1=0$ as  verified earlier.  When $n\equiv 2\pmod{4}$, by Lemma \ref{k=9lemma1}, the former is false, and we have the latter. Therefore we proved \ref{it:c10}.
\end{proof}

\section{On Primitive Elements of $H^*(BU_n;\mathbb{Z})$}\label{On Primitive}
Recall from Section 1 that the space $BU_n$, as the homotopy fiber of the map $BPU_n\rightarrow K(\mathbb{Z},3)$, admits an action by $BS^1\simeq K(\mathbb{Z},2)$:
\[\mu: K(\mathbb{Z},2)\times BU_n\rightarrow BU_n.\]

We have the following
\begin{proposition}[Proposition 3.2, \cite{To}, Toda, 1987]\label{coaction}
The ring homomorphism
\[\mu^*: H^*(BU_n;\mathbb{Z})\rightarrow H^*(K(\mathbb{Z},2);\mathbb{Z})\otimes_{\Z} H^*(BU_n;\mathbb{Z})\]
is determined by
\[\mu^*(c_k)=\sum_{i=0}^{k}v^i\otimes{n-k+i\choose i}c_{k-i}.\]
\end{proposition}
Recall from Section 1 that the primitive elements forms a subring of $H^*(BU_n;\mathbb{Z})$ denoted by $PH^*(BU_n;\mathbb{Z})$ which contains $\operatorname{Im}P^*$ as a subring, where $P: BU_n\rightarrow BPU_n$ is the map associated to the quotient map $U_n\rightarrow PU_n$.

Recall from Section \ref{The Higher Diff} that we have a homomorphism
\[\nabla: H^*(BU_n;\mathbb{Z})\rightarrow H^{*-2}(BU_n;\mathbb{Z})\]
given by Corollary \ref{diff4}:
\begin{equation}\label{nabla}
\nabla(c_k)=(n-k+1)c_{k-1}
\end{equation}
satifying the Leibniz rule. Then Proposition \ref{coaction} implies
\[\mu^*(x)=1\otimes x+v\otimes\nabla(x)+ \textrm{terms with higher orders of $v$}.\]
Therefore we have
\begin{equation}\label{imageprimitivenabla}
\operatorname{Im}P^*\subset PH^*(BU_n;\mathbb{Z})\subset\operatorname{Ker}\nabla=\operatorname{Ker}{^Ud}_3^{0,*}.
\end{equation}
This enables us to give the following
\begin{proof}[Proof of Theorem \ref{primitive}]
The computation in Section \ref{1<k<6}, \ref{k=7,8} and \ref{k=9,10} shows that $^Ud_r^{0,t}=0$ for all $r>3$ and $t\leq 8$. It remains to show $^Ud_r^{0,t}=0$ for $r>3$ and $t=10, 12$.

When $t=12$, for obvious degree reasons we only need to check the cases $r=9,11$ and $13$.

For $r=9, 11$, a routine computation of $^Ud_3^{*,*}$ shows that the sequences
\[{^UE}_3^{6,6}\xrightarrow{^Ud_3}{^UE}_3^{9,4}\xrightarrow{^Ud_3}{^UE}_3^{12,2}\]
\[{^UE}_3^{8,4}\xrightarrow{^Ud_3}{^UE}_3^{11,2}\xrightarrow{^Ud_3}{^UE}_3^{14,0}\]
are exact. Therefore we have ${^UE}_4^{9,4}=0$ and ${^UE}_4^{11,2}=0$. Therefore we have ${^Ud}_9^{0,12}=0$ and ${^Ud}_{11}^{0,12}=0$.

For $r=13$, notice that the group
\[{^UE}_2^{13,0}\cong H^{13}(K(\mathbb{Z},3);\mathbb{Z})\cong\mathbb{Z}/2\]
is generated by $x_1y_{2,1}$.

For $n$ odd, $x_1y_{2,1}$ is in the image of ${^Ud}_3^{10,2}$ and there is nothing else to prove.

For $n$ even, recall the map $\Delta': BPU_2\rightarrow BPU_n$ discussed in Section \ref{k=9,10}.
Let \[\rho_2: H^*(-;\mathbb{Z})\rightarrow H^*(-;\mathbb{Z}/2)\] be the canonical reduction. Then it follows from Lemma \ref{BSO_3} and equation (\ref{BSO_3mod2}) that we have $(\Delta')^*\rho_2(x_1)=\omega_3$ and $(\Delta')^*\rho_2(y_{2,1})=(\operatorname{Sq}^2(\omega_3))^2=\omega_2^2\omega_3^2$. Therefore we have
\[(\Delta')^*\rho_2(x_1y_{2,1})=\omega_2^2\omega_3^3\neq 0.\]
Hence there is no nontrivial differential into ${^UE}_*^{13,0}$. In particular, we have ${^Ud}_{13}^{0,12}=0$.

When $t=10$, the proof is very similar: we only need to check $^Ud_r^{0,12}$ for $r=9$ and $r=11$.

For $r=9$, again we have an exact sequence
\[{^UE}_3^{6,4}\xrightarrow{^Ud_3}{^UE}_3^{9,2}\xrightarrow{^Ud_3}{^UE}_3^{12,0}\]
which yields ${^UE}_4^{9,2}=0$. Therefore $^Ud_9^{0,10}=0$.

For $r=11$, the group
\[{^UE}_2^{11,0}\cong H^{11}(K(\mathbb{Z},3);\mathbb{Z})\cong\mathbb{Z}/3\]
is generated by $x_1y_{3,0}$.

For $3\nmid n$, one readily verifies $x_1y_{3,0}\in\operatorname{Im}{^Ud}_3^{8,2}$ and therefore ${^UE}_4^{11,0}=0$. Hence there is nothing more to show.

For $3|n$, we have a ``diagonal'' map
\[\Delta': BPU_3\rightarrow BPU_n\]
similar to the one considered in Lemma \ref{BSO_3}. In particular, we have
\begin{equation}\label{x1tox1}
\Delta'^*(x_1)=x_1.
\end{equation}
(Beware that the $x_1$'s on both sides of the equation are not the same, as they live in different groups, unless $n=3$.) Let
\[\rho_3: H^*(-;\mathbb{Z})\rightarrow H^*(-;\mathbb{Z}/3)\]
be the canonical reduction map, and $\mathscr{P}^i$ denote the $i$-th Steenrod power operation for $p=3$. Finally let $B$ be the Bockstein homomorphism. Then it is easy to deduce the following from Proposition \ref{K(Z,3)general}:
\[\rho_3(y_{3,0})=B\mathscr{P}^1(\rho_3(x_1)).\]
It follows from Theorem 4.11 of \cite{Ko1} (Kono, Mimura and Shimada, 1975) that we have $\rho_3(x_1)B\mathscr{P}^1(\rho_3(x_1))\neq 0$ for $n=3$. Equation (\ref{x1tox1}) shows that this is the case for all $n$ such that $3|n$. In other words, we have
$\rho_3(x_1y_{3,0})\neq 0$ for $3|n$. This shows that $x_1y_{3,0}$ is a nontrivial $3$-torsion in $^UE_{\infty}^{11,0}$. Therefore, it is not in the image of any differential. In particular, we have $^Ud_{11}^{0,10}=0$.

We have shown $\operatorname{Ker}{^Ud}_3^{0,t}\subset\operatorname{Im}P^*$ for $t\leq 12$. The theorem then follows from (\ref{imageprimitivenabla}).
\end{proof}

In fact, the assertion $PH^*(BU_n;\mathbb{Z})\subset\operatorname{Ker}\nabla$ may be improved. The following proposition is motivated by Crowley's consideration on primitive elements:
\begin{proposition}\label{primitivenabla}
As subrings of $H^*(BU_n;\mathbb{Z})$, the following holds:
\[PH^*(BU_n;\mathbb{Z})=\operatorname{Ker}\nabla.\]
\end{proposition}
\begin{proof}
It is straight forward to verify
\[\mu^*(c_k)=\sum_{i=0}^{k}v^i\otimes{n-k+i\choose i}c_{k-i}=\sum_{i=0}^{k}v^i\otimes\frac{\nabla^i}{i!}(c_k)=\operatorname{exp}(\widetilde{\nabla})(1\otimes c_k),\]
where
\[\widetilde{\nabla}: H^*(K(\mathbb{Z},2);\mathbb{Z})\otimes_{\Z} H^*(BU_n;\mathbb{Z})\rightarrow H^*(K(\mathbb{Z},2);\mathbb{Z})\otimes_{\Z} H^*(BU_n;\mathbb{Z})\]
is a homomorphism of abelian groups determined by
\begin{equation*}
\begin{cases}
\widetilde{\nabla}(1\otimes c_k)=v\otimes\nabla(c_{k-1}),\\
\widetilde{\nabla}(v\otimes1)=0,
\end{cases}
\end{equation*}
together with the Leibniz rule. It is formal to verify that $\operatorname{exp}(\widetilde{\nabla})$ is a ring homomorphism. Therefore we have $\mu^*=\operatorname{exp}(\widetilde{\nabla})$. For any homogeneous element $x\in H^*(BU_n;\mathbb{Z})$, clearly we have $\widetilde{\nabla}(1\otimes x)=0$ if and only if $\nabla(x)=0$, and the proposition follows.
\end{proof}
We proceed to offer a way to study $^UE_*^{*,*}$ by using known results on primitive elements. It follows from Corollary 4.3 and Section 4.5 of \cite{To} (Toda, 1987) that we have
\begin{equation}\label{Toda}
\operatorname{Im}P^*=PH(BU_n;\mathbb{Z}/2)
\end{equation}
for $n\equiv 2\pmod{4}$ and $n=4$.

Let $^UE_{*}^{*,*}(\mathbb{Z}/2)$ be the spectral sequence $^UE_{*}^{*,*}$ with coefficient ring $\mathbb{Z}$ replaced by $\mathbb{Z}/2$. Then we have the following
\begin{lemma}\label{E(Z2)}
For $t>0$, the entry $^UE_{*}^{9,t}(\mathbb{Z}/2)$ is the target of a nontrivial differential if and only if $^UE_{*}^{9,t}$ is.
\end{lemma}
\begin{proof}
Recall that $H^*(K(\mathbb{Z},3);\mathbb{Z}/2)$ is the polynomial algebra generated by the elements $\operatorname{Sq}^I(\rho_2(x_1))$ where $x_1$ is the fundamental class and $I$ an admissible sequence of excess$<3$. (See Chapter 3 and 9 of \cite{Mo}, Mosher and Tangora, 1968, for details.) Moreover, the subalgebra $\operatorname{Im}\rho_2$ is generated by $\rho_2(x_1)$ and $[\operatorname{Sq}^I(\rho_2(x_1))]^2$ where $I\neq\phi$. Then in degree less than $9$, the only elements of $H^*(K(\mathbb{Z},3);\mathbb{Z}/2)$ not in $\operatorname{Im}\rho_2$ are $\operatorname{Sq}^2(\rho_2(x_1))$ and $\rho_2(x_1)\operatorname{Sq}^2(\rho_2(x_1))$, of degree $5$ and $8$ respectively. The lemma then follows for obvious degree reasons.
\end{proof}
We have the following proposition regarding the differentials of $^UE_{*}^{*,*}$
\begin{proposition}
For $n\equiv 2\pmod{4}$ and $n=4$, $^Ud_{9}^{0,t}=0$ for all $t>0$.
\end{proposition}
\begin{proof}
It follows from ((\ref{Toda}), Toda, 1987) that this is true in the spectral sequence $^UE_{*}^{9,t}(\mathbb{Z}/2)$. The proposition then follows from Lemma \ref{E(Z2)}.
\end{proof}
\appendix
\section{Multiplicative Constructions and Twisted Tensor Products}
 This appendix is a review of part \emph{S{\'e}minaire Henri Cartan} required for Section \ref{DG Alg}, with twisted tensor products (\cite{Br}, Brown, 1959) as an example. The author makes no claim of originality to materials in this appendix. As in the introduction, all DGA's involved are graded-commutative and augmented over the base ring $R$, which is either $\mathbb{Z}$ or $\mathbb{Z}/p$ for some prime $p$.

\begin{definition}\label{mult}
 A \emph{construction} is a triple $(A, N, M)$ where $A$ is a DGA and $N$ is a DG module, both with augmentations over $R$, and $M$ is a chain complex with a graded algebra structure, satisfying the following conditions:
 \begin{enumerate}
 \item As a graded algebra (not necessarily as a chain complex), we have $M=A\otimes_R N$.
 \item The  differential of $N$ is determined by that of $M$,  via the relation $N=R\otimes_{A}M$ where $A$ acts on $R$ by augmentation.
 \end{enumerate}
An \emph{acyclic construction} is one such that $M$ is acyclic. When $N$ is also a $DGA$, we say that $(A, N, M)$ is a \emph{multiplicative construction}.
\end{definition}

\begin{remark}
 $A$ (resp. $N$) is a sub-DGA of $M$ via $a\mapsto a\otimes 1$ (resp. $n \mapsto 1\otimes n$). We will use this fact implicitly.
\end{remark}

\begin{example}\label{bar}
 Let $A$ be a DGA, and $\epsilon$ be its augmentation. Let $\bar{A}=\operatorname{Ker}\epsilon$. The \emph{bar construction} of $A$ is a multiplicative construction $(A, \overline{\mathscr{B}}(A), \mathscr{B}(A))$ where the DGA's $\mathscr{B}(A)$ and $\overline{\mathscr{B}}(A)$ are defined as follows:

 \begin{enumerate}
 \item \[\mathscr{B}(A)=\sum_{k\geq 0}A\otimes_{R} \bar{A}^{\otimes k}; \overline {\mathscr{B}}(A)=\sum_{k\geq 0}\bar{A}^{\otimes k}.\]
  By convention we have $\bar{A}^{\otimes 0}=A^{\otimes 0}=R$. For simplicity the element $a\otimes a_1\otimes \cdots \otimes a_k$ of $\mathscr{B}(A)$ is denoted by
  \[a[a_1|\cdots |a_k]\]
  and
  \[1 \cdot [a_1|\cdots |a_k]=[a_1|\cdots |a_k].\]

 \item The degree is defined by
  \[\operatorname{deg}(a[a_1|\cdots |a_k])=k+\operatorname{deg}(a)+\operatorname{deg}(a_1)+ \cdots +\operatorname{deg}(k) \]
  and the length of $a[a_1|\cdots |a_k]$ is defined to be $deg([a_1|\cdots |a_k])$. We define a bi-degree on $\mathscr{B}(A)$ by saying that $a[a_1|\cdots |a_k]$ has bi-degree $(s,t)$ if deg$(a)=s$, deg$([a_1|\cdots |a_k])=t$. The degree of $a[a_1|\cdots |a_k]$ is obviously $s+t$. Let $\mathbf{F}_{B}$ be the filtration induced by the first entry $s$ of this bi-degree.

 \item A chain map $s: \mathscr{B}(A) \rightarrow \overline {\mathscr{B}}(A)$ of degree one is defined as follows:
  \[s(a[a_1|\cdots |a_k])=[a|a_1|\cdots |a_k] s(1)=0\]

 \item The differential of $\mathscr{B}(A)$ is defined by induction on $k$ as follows:
  \[
    \begin{cases}
     d([a])=a\cdot 1-[d(a)]-\eta(a)\cdot 1\\
     d([a_1|\cdots|a_k])=a_1[a_2|\cdots|a_k]-sd(a_1[a_2|\cdots|a_k]),k\geq 2
    \end{cases}
  \]
  and the differential of $\overline{\mathscr{B}}(A)$ is induced by the above.
 \item $\overline{\mathscr{B}}(A)$ has a product structure induced by that of $A$, sometimes called the shuffle product, which makes $\overline{\mathscr{B}}(A)$ a DGA with respect to the differential defined above.

 \end{enumerate}
It is obvious and proved in Exp.3 of \cite{Ca} (Cartan and Serre, 1954-1955) that $s$ is a chain homotopy between the identity of $\mathscr{B}(A)$ and the augmentation $\eta$. We can iterate this procedure to construct inductively $\mathscr{B}^{n+1}(A)=\mathscr{B}(\overline{\mathscr{B}}^{n}(A))$ and $\overline{\mathscr{B}}^{n+1}(A)=\overline{\mathscr{B}}(\overline{\mathscr{B}}^{n}(A))$
In the case $A=R[\Pi]$(concentrated in degree 0) where $\Pi$ is an abelian group, we have $H_{*}(\overline{\mathscr{B}}^{n}(R[\Pi]))\cong H_{*}(K(\Pi, n);R)$. This is shown in, for example, \cite{Ei} (Eilenberg and Mac Lane, 1953) and \cite{Mi} (Milgram, 1967).
\end{example}

\begin{example}\label{twisted tensor product}[\cite{Br}, Brown, 1959]
Fix a base ring $R$. Let $K$ be a DG coalgebra and $A$ a DGA. In \cite{Br}, Brown defined the notion of a twisted cochain, i.e., a cochain $\varphi\in C^*(K;A)$ satisfying certain axioms. Now given a differential graded $A$-module $L$, the twisted tensor product $K_{\varphi}\otimes L$ is a chain complex, of which the underlying $R$-module is that of the graded algebra $K\otimes L$. Then $(A, K, K_{\varphi}\otimes A)$ is a multiplicative construction.
\end{example}

The differential of $K_{\varphi}\otimes L$ admits a filtration $\mathbf{G}$ as follows:
\begin{equation}\label{filtrationG}
\mathbf{G}_s(K_{\varphi}\otimes L)=\sum_{i=0}^{s}(K_s\otimes L),
\end{equation}
which induces a spectral sequence $E^*_{*,*}$ such that
\begin{equation}\label{Gspectralseq}
E^2_{s,t}\cong H_s(K;H_t(L)).
\end{equation}
For a map of pointed spaces $\pi: X\rightarrow B$ admitting a weakly transitive function ($\pi$ a Serre fibration, for example), let $F=\pi^{-1}(b_0)$, where $b_0$ is the base point of $B$. Brown defined the loop space $\Omega B$ differently from the convention, such that it is homotopy equivalent to the conventional one, but the composition of loops yields an associative product, instead of that in the conventional case where the product is associative only up to homotopy. Therefore, the based singular complex $S_*(\Omega B)$ is a DGA, and $S_*(F)$ is a $S_*(\Omega B)$ module.
Let $S^{(1)}_*(B)$ be the first Eilenberg sub-complex of $B$ (\cite{Ei1}, Eilenberg, 1944), i.e., the sub-complex of $S_*(B)$ of singular simplices sending all vertices of a standard simplex to $b_0$. Finally let $S^F_*(X)$ be the sub-complex of $S_*(X)$ of singular simplices sending all vertices of a standard simplex to $F$.
Then we have the following
\begin{proposition}\label{Brown}[E. H. Brown, Jr, (4.4) Corollary and (7.2) Corollary of \cite{Br}, Brown, 1959]
There is a twisted cochain $\Phi_B\in C^*(S^{(1)}_*(B), S_*(\Omega B))$ such that there is a chain homotopy equivalence
\begin{equation*}
S^{(1)}_*(B)_{\Phi_B}\otimes S_*(F)\rightarrow S^F_*(X).
\end{equation*}
Moreover, the spectral sequence induced by the filtration $\mathbf{G}$ is isomorphic to the Serre spectral sequence of the fibration $\pi:X\rightarrow B$.
\end{proposition}

We proceed to introduce the operations on an acyclic multiplicative construction $(A,N,M)$ with base ring $R$, following Section 6 to 12 of \cite{Ca} (Cartan and Serre, 1954-1955). Let $d_A, d_N, d_M$ be the differentials of $A,N$ and $M$ respectively. Furthermore, we assume that there are augmentations $\epsilon_{A}, \epsilon_{N}, \epsilon_{M}$ from $A,N,M$ respectively, to the base ring $R$. The subscripts will be omitted whenever there is no risk of ambiguity.

\begin{definition}\label{suspension}
 Assume that the homomorphism $A\rightarrow M: a\mapsto a\otimes 1$ is injective. Let $\alpha\in H_{k}(A)$ be represented by $a\in\operatorname{Ker}d_{A}$. Since M is acyclic, there is some $x\in M$ such that $d_{M}(x)=a$. Passing to $N$, we obtain an element
 \[\bar{x}=1\otimes x\in A\otimes_{R}M\cong N.\]
 where $\bar{x}$ is easily verified to be a cycle in $N$. The homology class $\{\bar{x}\}\in H_{k+1}(N)$ is therefore called the  \emph{suspension } of the homology class $\alpha\in H_{k}(A)$, denoted by $\sigma(\alpha)$. It is easy to verify that $[\bar{x}]$ is independent of the choice of $a$ or $x$ and $\sigma: H_{k}(A)\rightarrow H_{k+1}(N)$ is a well defined homomorphism of graded abelian groups of degree $1$.
\end{definition}

\begin{example}\label{suspensionex}
If $(A,N,M)=(A,\overline{\mathscr{B}}(A),\mathscr{B}(A))$, then the suspension can be realized by the homomorphism $A\rightarrow \mathscr{B}(A), a\mapsto [a]$. Notice that the presence of the bracket lifts the degree by $1$.
\end{example}

\begin{definition}\label{transpotence}
 Let the character of the base ring $R$ be a prime number $p$. Define the following $R$-submodule of of $A_{2q}$ by
 \[{_{p}A_{2q}}=\{a\in A_{2q}|d_{A}(a)=0,a^{p}=(\epsilon(a))^p\}.\]
 Take $x\in M_{2q+1}$ such that $d_{M}(x)=a-\epsilon(a)$. Then we have $d_{M}((a-\epsilon(a))^{p-1}x)=(a-\epsilon(a))^{p}=0$. Now take $y\in M_{2pq+2}$ such that $d(y)=(a-\epsilon(a))^{p-1}x.$ Passing to $N$ and taking homology we define the \emph{transpotence} $\psi_p(a)=\bar{y}\in H_{2pq+2}(N)$.
\end{definition}

Notice that $\psi_p: {_{p}A_{2q}}\rightarrow H_{2pq+2}(N)$ is not necessarily a homomorphism, even of abelian groups, and does not necessarily pass to homology. However we have the following

\begin{proposition}(H. Cartan, Proposition 5, exp. 6, \cite{Ca}, Cartan and Serre, 1954-1955)
 Suppose we have
 \begin{enumerate}
 \item
   $a^{p}=0,$ for all $a\in A_{2q}$, and
 \item
  $b\cdot d_{A}(b^{p-1})$ is in the image of $d_{A}$, for all $b\in A_{2q+1}$.
 \end{enumerate}
 Then $\psi_p$ passes to homology to define a map $\varphi_p: H_{2q}(A)\rightarrow H_{2pq+2}(N).$
\end{proposition}

This $\psi_p$ is again not necessarily additive. But in the case that interests us, we have the following

\begin{proposition}(H. Cartan, Theorem 3, exp. 6, \cite{Ca}, Cartan and Serre, 1954-1955)
 Let $A$ be a commutative $R$-algebra, regarded as a graded-commutative DGA concentrated in degree $0$ over a base ring $R$ of characteristic $p$ where $p$ is a prime. Then $\psi_p: {_{p}A}\rightarrow H_{2}(\overline{\mathscr{B}}(A))$ is additive when $p$ is odd. For all $n\geq 1$ and $q\geq 1$, the transpotence $\varphi_p: H_{2q}(\overline{\mathscr{B}}^{n}(A))\rightarrow H_{2pq+2}(\overline{\mathscr{B}}^{n+1}(A))$ induced by $\psi_p$ is well defined, and if $p$ is odd, it is additive with kernel containing all the decomposable elements of $H_{2q}(\overline{\mathscr{B}}^{n}(A)).$
\end{proposition}

\begin{definition}\label{divided power}
 For $A$ graded-commutative, a \emph{divided power operation} on $A$ is a collection of maps $\gamma_{k}:A\rightarrow A$ for all integers $k\geq 0$, such that for any $x,y\in A$ we have the following axioms:
 \begin{enumerate}
 \item
  $\gamma_{0}(a)=1, \gamma_{1}(a)=a, \textrm{deg}\gamma_{k}(a)=k \textrm{deg}a.$
 \item
  $\gamma_{k}(x)\gamma_{l}(x)={k+l \choose k}\gamma_{k+l}(x).$
 \item
  (Leibniz rule) $\gamma_{k}(x+y)=\sum_{i+j=k}\gamma_{i}(x)\gamma_{j}(y).$
 \item
 $$\gamma_{k}(xy)=
    \begin{cases}
    0 ,\textrm{deg}(x), \textrm{deg}(y) \textrm{ are odd}, k\geq 2,\\
    x^{k}\gamma_{k}(y), \textrm{deg}(x), \textrm{deg}(y) \textrm{ are even}, \textrm{deg}(y)\geq 2.
    \end{cases}
 $$
  When the characteristic of $R$ is 2, we have in addition, for $k\geq 2$,
 $$\gamma_{k}(xy)=
    \begin{cases}
    0 ,\textrm{deg}(x), \textrm{deg}(y)>0,\\
    x^{k}\gamma_{k}(y), \textrm{deg}(x)=0.
    \end{cases}
 $$
 \item
  $\gamma_{k}(\gamma_{l}(x))={2l-1 \choose l-1}{3l-1 \choose l-1}\cdots {kl-1 \choose l-1}\gamma_{kl}(x).$
 If in addition, $A$ is a DGA with differential $d$, then we require
 \item
  $d(\gamma_{k}(x))=\gamma_{k-1}(x)d(x)$ for $k\geq 1$.
 \end{enumerate}
 A graded-commutative algebra with a divided power operation is called a \emph{divided power algebra}. A map of divided power algebras is a homomorphism of graded algebras compatible with the divided power operation.
\end{definition}

\begin{remark}
 We do not require $A$ to have a differential, since we often wish to define a divided power operator on the homology of a DGA, rather than the DGA itself.
\end{remark}

\begin{example}\label{divided power ex1}
  The prototype of a divided power algebra is $P_{R}(y)$, which, as a graded $R$-algebra, is generated  by element $\gamma_{k}(y)$ for all $k\geq 1$ modulo the relations imposed by definition \ref{divided power}. Here $y$ is of degree $2q$ for any positive integer $q$. By (4) of definition \ref{divided power}, we have $k!\gamma_{k}(y)=y^k$. In fact, when $R$ is torsion-free, $P_{R}(y)$ is isomorphic to the polynomial algebra $R[y]$ adjoining all $\frac{y^k}{k!}=\gamma_{k}(y).$

  On the other hand, if $R=\mathbb{Z}/p$ where $p$ is a prime number, then (4) of definition \ref{divided power} implies that $y^p=p!\gamma_{p}(y)=0$. Furthermore, for $k=k_{0}+k_{1}p+k_{2}p^2+\cdots +k_{r}p^r$, where $0\leq k_{i}<p-1, i=0,1,\cdots, r$, we have $\gamma_{k}(y)=\prod_{0\leq i<r}\gamma_{k_i}(\gamma_{p^i}(y))=\gamma_{p^i}(y)/k_{i}!$. In fact, as a graded $\mathbb{Z}/p$-algebra

  \begin{equation}\label{divided power fomula}
  P_{\mathbb{Z}/p}(y)\cong \bigotimes_{k\geq 0} \mathbb{Z}/p[\gamma_{p^k}(y)].
  \end{equation}

  A detailed discussion on divided power algebras over $\mathbb{Z}/p$, including the proofs of the statements above, can be found in section 7, exp. 7 of \cite{Ca} (Cartan and Serre, 1954-1955).
\end{example}

\begin{example}\label{divided power ex2}
\
 \begin{enumerate}
 \item
 Let $M$ be a free graded $R$-module generated by elements in odd degrees. Then we can form the free exterior $E(M)$ over a given set of generators and take the trivial divided power operations such that $\gamma_{0}$ is the constant map on to $1\in R$, $\gamma_{1}$ is the identity, and $\gamma_{k}(x)=0$ for $k>1$. This is considered usually when $R$ is a field of charcter other than $2$.
 \item
 Let $M$ be a free graded $R$-module generated by a set $\{e_{i}\}$ of elements. If $R$ has characteristic other than $2$, we require $e_i$ to have even degrees. Then we can form the  the universal symmetric algebra $S(M)$ which has the underlying graded $R$-module $\sum_{i\geq 0}M^{\otimes i}$, with $M^{\otimes 0}=R$ in degree 0, and with a structure of DGA such that for any divided power algebra $A$, an $R$-linear map $M\rightarrow A$ can be extended uniquely to a map of divided power algebra $S(M)\rightarrow A$.

 The product on $S(M)$ is shuffle product similar to that of the reduced bar construction. In particular $e_{i}^{k}=k!e_{i}\otimes\cdots\otimes e_{i}$ with $k$ copies of $e_{i}$'s on the righthand side of the equation.       There is a power operation $\{\gamma_{k}\}_{k}$ on $S(M)$ determined by $\gamma_{k}(e_{i})=e_{i}\otimes\cdots\otimes e_{i}$ with $k$ copies of $e_{i}$ and the axioms in Definition \ref{divided power}.

 \item
 Let $A$ be a DGA and $a\in A$. Then $[a]\in\overline{\mathscr{B}}(A)$, and the shuffle product $[a]^k=k![a|\cdots|a]$ with $k$ copies of $a$ on the righthand side of the equation. We can define a divided power operation on $\overline{\mathscr{B}}(A)$ similar to that of $S(M)$, by requiring $\gamma_{k}([a])=[a|\cdots|a]$ with $k$ copies of $a$ in the bracket. For more details see section 4, exp. 8 of \cite{Ca} (Cartan and Serre, 1954-1955).
 \end{enumerate}
\end{example}

\begin{example}\label{divided power ex3}
 Let $A$ and $A'$ be graded-commutative DGA's with divided power operations. Then (3) and (4) of Definition \ref{divided power} gives a unique way to extend the divided power operations to $A\otimes A'$, a DGA with differential determined by those of $A$ and $A'$ via the Leibniz rule. For details see Theorem 2, exp. 7 of \cite{Ca} (Cartan and Serre, 1954-1955).
\end{example}

Let $M$ be a free graded $R$-module and write $M=M_+\oplus M_-$, where $M_+$ (resp. $M_-$) is the direct sum of $M_k$'s for even (resp. odd) $k$. Let
\begin{equation*}
U(M)=
\begin{cases}
S(M), \textrm{ $R$ has character $2$},\\
S(M_+)\otimes E(M_-), \textrm{ otherwise}.
\end{cases}
\end{equation*}
Then in the sense of Example \ref{divided power ex2} and Example \ref{divided power ex3}, $U(M)$ is a graded-commutative algebra with a divided power operation. We call $U(M)$ the universal algebra of $M$ in the sense of the following
\begin{theorem}[H. Cartan, Theorem 2, exp.8, \cite{Ca}, Cartan and Serre, 1954-1955]\label{U(M)}
 Let $A$ be a graded-commutative algebra with a divided power operation as in Definition \ref{divided power} and $M$ a free graded module. Then any homomorphism of graded $R$-modules $f: M\rightarrow A$ can be extended uniquely to a map of divided power algebras $U(M)\rightarrow A$.
\end{theorem}

In the case $p=2$ we have the following relation of the three operations:

\begin{proposition}[Proposition 1, exp.8, \cite{Ca}, Cartan and Serre, 1954-1955]\label{relation}
For $q\geq 1$, and $A$ a graded-commutative DGA, we have
\[
  \varphi_2=\gamma_2\cdot\sigma: H_{2q}(A)\rightarrow H_{4q+2}(\bar{\mathscr{B}}(A)).
\]
\end{proposition}
The three operations, the suspension, the transpotence, and the divided power operation, as described above, are enough to describe the homology of $K(\Pi,n)$ with coefficients in $\mathbb{Z}$ or $\mathbb{Z}/p$, at least when $\Pi$ is a finitely generated abelian group. We start with the definitions of words and their heights and degrees.

\begin{definition}\label{words}
 We give parallel definitions in the cases when $p$ is odd and $p=2$.
 \begin{enumerate}
 \item
 By a \emph{word} we mean a sequence consists of the three symbols $\sigma, \varphi_{p},$ and $\gamma_{p}$ when $p$ is odd, or $\sigma,\gamma_{2}$ when $p=2$, where repetition is allowed. The \emph{height} of a word $\alpha$ is the total number of $\sigma$ and $\varphi_{p}$ in $\alpha$, counting repetition. We take the degree of the empty word to be $0$ and inductively define the degree of the words $\sigma\alpha, \varphi_{p}\alpha$ and $\gamma_{p}\alpha$ as follows:
 \begin{equation}
  \begin{cases}
  \operatorname{deg}(\sigma\alpha)=\operatorname{deg}(\alpha)+2,\\
  \operatorname{deg}(\varphi_{p}\alpha)=2p\operatorname{deg}(\alpha)+2,\\
  \operatorname{deg}(\gamma_{p}\alpha)=p\operatorname{deg}(\alpha).
  \end{cases}
 \end{equation}
 \item
 When $p$ is odd, an \emph{admissible word}, is a word $\alpha$ such that (i) $\alpha$ is non-empty and starts and ends with $\sigma$ or $\varphi_{p}$, and (ii) for every $\varphi_{p}$ and $\gamma_{p}$ appeared in $\alpha$, there are even number of copies of $\sigma$ on its righthand side. An admissible word is of type 1 if it ends with $\sigma$, and type 2 if it ends with $\varphi_{p}$.

 \item
 When $p=2$, an \emph{admissible word}, is a word $\alpha$ that starts and ends with $\sigma$. $\alpha$ is of type 2 if it ends with $\gamma_{2}\sigma$, and of type 1 otherwise.
\end{enumerate}
 The words consisting of $\sigma, \varphi_{p}$ and $\gamma_{p}$ are also called $p$-words.
\end{definition}

A word can be regarded as the compose of the sequence of operations $\sigma, \varphi_{p}$ and $\gamma_{p}$, in the obvious manner. Let $\mathbb{Z}[\Pi]$ be the group ring of a finitely generated abelian group $\Pi$, viewed as a DGA concentrated in degree $0$ and with a trivial differential. Then $H_{*}(\mathbb{Z}[\Pi])\cong\mathbb{Z}[\Pi]$. The alert reader will find that, an admissible word $\alpha$ is a well defined compose of operations on $\mathbb{Z}[\Pi]$, with image in $H_{\operatorname{deg}(\alpha)}(\overline{\mathscr{B}^{n}}(\mathbb{Z}[\Pi]))\cong H_{\operatorname{deg}(\alpha)}(K(\Pi,n))$, where $n$ is the height of $\alpha$. In fact all homology classes are generated this way, as we will see soon.

\begin{definition}\label{U(M)p}
 Let $\Pi$ be a finitely generated abelian group, and let $_{p}\Pi$ be the subgroup of $\Pi$ of elements of order infinity or a power of $p$. We write $\Pi/(p\Pi)=\prod_{i}\Pi'_{i}$ and $_{p}\Pi=\prod_{j}\Pi''_{j}$ as the decomposition of $\Pi/(p\Pi)$ and $_{p}\Pi$ into direct products of cyclic groups of order infinity or a power of $p$.

 Fix a positive integer $n$. Let $M^{(n)}$ be the free graded $R$-module generated by $\alpha_{i}$ of degree deg($\alpha$) for every admissible word $\alpha$ of type 1 with height $n$ and $\Pi'_{i}$, and $\alpha'_{j}$ of degree deg($\alpha'$) for every admissible word $\alpha'$ with height $n$ of type 2. Let $U(M^{(n)})$ be as in Proposition \ref{U(M)}. Notice in particular that $U(M^{(0)})\cong\mathbb{Z}/p[\Pi]$.
\end{definition}

In the following theorem we do not distinguish a word and the compose of operations that it represents.

\begin{theorem}[H. Cartan, Th{\'e}or{\`e}me fondamental, exp. 9, \cite{Ca}, Cartan and Serre, 1954-1955]\label{fundamental}
 With the notations as above, let $w'_{i},w''_{j}$ be generators of $\Pi'_{i},\Pi''_{j}$ respectively. Take $R=\mathbb{Z}/p$, where $p$ is an odd prime number. Let $f^{(n)}:M^{(n)}\rightarrow H_{*}(K(\Pi,n);\mathbb{Z}/p)$ be the homomorphism of $\mathbb{Z}/p$-modules taking $\alpha_{i}$ (resp. $\alpha'_{j}$) to $\alpha(w'_{i})$ (resp. $\alpha(w''_{j})$). Its unique extension to $U(M^{(n)})$, $\tilde{f}^{(n)}$ given by Proposition \ref{U(M)}, is an isomorphism of divided power algebras.
\end{theorem}

We will give a sketch of proof of Theorem \ref{fundamental} since the idea is relevant to our application. To do so we need the following theorems.

\begin{theorem}[H. Cartan, Theorem 2, exp.2, \cite{Ca}, Cartan and Serre, 1954-1955]\label{Ca1}
  Let $f: A \rightarrow A'$ be a morphism of DGA's over $R$. Let $M$ (resp. $M'$) be an acyclic chain complex over $R$ with a graded $A$ (resp.$A'$)-module structure. Let $I$ (resp. $I'$) be the kernel of the augmentation of $A$ (resp. $A'$). Then there is a morphism of chain complexes $g: M/IM \rightarrow M'/I'M'$ compatible with $f$ in the obvious sense. The induced morphism $H_{*}(g)$ is independent of the choice of $g$. Moreover, if $f$ is a weak equivalence, then so is $g$.
\end{theorem}

\begin{theorem}[H. Cartan, Theorem 5, exp.4, \cite{Ca}, Cartan and Serre, 1954-1955]\label{Ca2}
  Let $(A,N,M)$ and $(A',N',M')$ be two multiplicative constructions. Let $\tilde{N'}$ be an $R$-subalgebra of $M$ containing $N$ such that $d: \tilde{N'}_{k+1} \rightarrow \operatorname{Ker}(d_k)$ is a degree-wise isomorphism of $R$-modules for all $k\geq 0$. Here $d_0=\epsilon$. In particular, $M'$ is acyclic. Let $f: A \rightarrow A'$ be a map of DGA's. Then there is a unique map $g: M \rightarrow M'$ of DGA's restricting to
 $f$, such that $g(N)\subset \tilde{N'}$.
\end{theorem}
\begin{proof}[Sketch of proof of Theorem \ref{fundamental}]
 By K{\"u}nneth formula it suffice to consider the case that $\Pi$ is a cyclic group with a generator $w$. We proceed to show that there is a multiplicative construction
 \begin{equation}\label{UMn}
 (U(M^{(n)}),U(M^{(n+1)}), L)
 \end{equation}
  with $L$ acyclic.

One can easily show that an admissible word $\alpha$ of height $n+1$ is of the form
 \begin{enumerate}
 \item
  $\sigma\alpha'$ where $\alpha'$ is of height $n$ and odd degree, or
 \item
  $\sigma\gamma_{p^k}\alpha'$ or $\varphi_{p}\gamma_{p^k}\alpha'$ where $k\geq 0$ and $\alpha'$ is of height $n$ and even degree.
 \end{enumerate}
 The base case where $n=0$ is easy. Let $L=U(M^{(n)})\otimes_{\mathbb{Z}/p}U(M^{(n+1)})$ as a graded $\mathbb{Z}/p$-algebra, and consider $U(M^{(n)})$ and $U(M^{(n+1)})$ as its subalgebras in the obvious manner.

 In the first case as above, let $x=\alpha'(w)\in U(M^{(n)})$. Then the free exterior algebra $E_{\mathbb{Z}/p}(x)$ is a subalgebra of $U(M^{(n)}$. Let $y=\alpha(w)=\sigma\alpha'(w)\in U(M^{(n+1)})$, and we have the subalgebra $P_{\mathbb{Z}/p}(y)$ of $U(M^{(n+1)})$. Define the differentials of $x$ and $y$ in $L$ by $d_{L}(x)=0$ and $d_{L}(y)=x$ together with the axioms in Definition \ref{divided power}. Then the $E_{\mathbb{Z}/p}(x)\otimes P_{\mathbb{Z}/p}(y)$ is acyclic.

 In the second case, let $x=\alpha'(w)\in U(M^{(n)}), y_{k}=\sigma\gamma_{p^k}(w)$ and $z_{k}=\phi_{p}\gamma_{p^k}(w)$ for all $k\geq 0$. Define their differentials in $L$ by $d_{L}(x)=0, d_{L}(y_k)=x$ and $d_{L}(z_{k})=x^{p-1}y_{k}$. Then one can show that
 $$P_{\mathbb{Z}/p}(x)\otimes\bigotimes_{k\geq 0}E_{\mathbb{Z}/p}(y_k)\otimes\bigotimes_{k\geq 0}P_{\mathbb{Z}/p}(z_k)$$
 is acyclic.

 By Theorem \ref{Ca1} and Theorem \ref{Ca2} one can inductively prove the statement, the base case where $n=1$ being standard homological algebra.
\end{proof}

\begin{remark}
 This argument fails for $p=2$, in which case $\varphi$ is not additive.
\end{remark}

For a free graded $R$-module $M$, recall the graded $R$-algebra $S(M)$ introduced in Example \ref{divided power ex2}, (2). We take $M^{(n)}$ as in Definition \ref{U(M)p}. and $f^{(n)}$ as in Theorem \ref{fundamental}. Notice that the constructions apply to $p=2$. The analog of Theorem \ref{fundamental} in the case where $p=2$ is the following:

\begin{theorem}[H. Cartan, Th{\'e}or{\`e}me fondamental, exp. 9, \cite{Ca}, Cartan and Serre, 1954-1955]\label{fundamental2}
 $f^{(n)}:M^{(n)}\rightarrow H_{*}(K(\Pi,n);\mathbb{Z}/2)$ extends to $\tilde{f}^{(n)}:S(M^{(n)})\rightarrow H_{*}(K(\Pi,n);\mathbb{Z}/2)$ which is an isomorphism of divided power algebras.
\end{theorem}
The proof is similar to that of Theorem \ref{fundamental}.

For the following corollary, recall the notations from Theorem \ref{Brown}. In particular, we may regard the based singular chain complex $S_*(K(\Pi,n);R)$ as a DGA over $R$. Also recall the multiplicative construction (\ref{UMn}) and let $\mathbf{F}_U$ be the increasing filtration on $L$ given by
\[\mathbf{F}_U^t(L)=\sum_{s\geq0,k\leq t}U(M^{(n)})_s\otimes_{\mathbb{Z}/p}U(M^{(n+1)})_k.\]
\begin{corollary}\label{equivofFiltration}
For $n\geq 1$ and each prime $p$, the filtration $\mathbf{F}_U$ induces a spectral sequence isomorphic to the Serre spectral sequence associated to the path fiberation
\begin{equation*}
K(\Pi,n)\rightarrow PK(\Pi,n+1)\rightarrow K(\Pi, n+1).
\end{equation*}
\end{corollary}
\begin{proof}
In Section 10 of \cite{Br1} (Brown, 1959), Brown defined a map of DG modules
\[\varphi^B: S^{(1)}_*(Y)\rightarrow \bar{\mathscr{B}}S_*(\Omega Y)\]
for a paracompact space $Y$, which is also a week equivalence. The curious readers may easily check that when $Y=\Omega X$ for some paracompact space $X$, $\varphi^B$ is a map of DGAs. Therefore, for $Y=K(\Pi,n)$, we have a multiplicative construction $(\bar{\mathscr{B}}S_*(K(\Pi,n-1)), \bar{\mathscr{B}}S_*(K(\Pi,n-1))_\Phi\otimes \bar{\mathscr{B}}S^{(1)}_*(K(\Pi,n+1)),\bar{\mathscr{B}}S^{(1)}_*(K(\Pi,n+1)))$ where the middle term is a twisted tensor product, quasi-isomorphic to $S^F_*(PK(\Pi,n+1))$. Here $F=K(\Pi,n)$.

It follows from Corollary 31.3 of \cite{Ei1} (Eilenberg, 1944) that $S^{(1)}_*(K(\Pi,n+1))$ is chain equivalent to $S_*(K(\Pi,n+1))$. Similarly one can show that $S^F(PK(\Pi,n+1))$ is chain equivalent to $S_*(PK(\Pi,n+1))$. In particular, this means that the twisted tensor product $\bar{\mathscr{B}}S_*(K(\Pi,n-1))_\Phi\otimes \bar{\mathscr{B}}S^{(1)}_*(K(\Pi,n+1))$ is acyclic.

Therefore, it follows from Proposition \ref{Brown} that the bi-degree on the twisted tensor product induces the Serre spectral sequence of the path fibration in the corollary. It follows from Theorem \ref{Ca2} that it suffices to construct a map of DGA's with divided power operations
\[U(M^{(n)})\rightarrow \bar{\mathscr{B}}S^{(1)}_*(\Omega K(\Pi,n+1))\]
which is a weak equivalence. It follows from Theorem \ref{fundamental} and Theoerm \ref{fundamental2} that such a chain equivalence exists, and in particular, sends each $p$-words $\alpha$ to a cycle that represents the same homology class as $\alpha$. Therefore we have a map of graded modules
\[M^{(n)}\rightarrow \bar{\mathscr{B}}S^{(1)}_*(\Omega K(\Pi,n+1)).\]
By Proposition \ref{U(M)} this extends to the desired DGA map.
\end{proof}
We proceed to consider the integral homology of $K(\Pi,n)$. Recall the transpotence $\psi_p$ defined in Definition \ref{transpotence}. For an arbitrary integer $l$, we have a similar operation $\psi_{l}:{_{l}\Pi}\rightarrow H_{2}(K(\Pi,1);\mathbb{Z}/l)$ where ${_{l}\Pi}$ is the subgroup of $\Pi$ of $l$-torsion elements. $\psi_{l}$ satisfies the following condition.

\begin{proposition}\label{psil}
 Let $\delta_{l}:H_{2}(K(\Pi,1);\mathbb{Z}/l)\rightarrow H_{1}(K(\Pi,1);\mathbb{Z})$ be the Bockstein homomorphism and $\sigma:{_{l}\Pi} \rightarrow H_{1}(K(\Pi,1);\mathbb{Z})$ be the suspension. Then $\sigma=\delta_{l}\psi_{l}$.
\end{proposition}
For details see section 1, exp. 11 of \cite{Ca} (Cartan and Serre, 1954-1955). In the case of integral cohomology, we extend the definition of an admissible $p$-word as follows.

\begin{definition}\label{wordintegral}
 A $p$-word is a sequence consists of the symbols $\sigma, \varphi_{p}, \gamma_{p}$, and $\psi_{p^{\lambda}}$, for some positive integer $\lambda$. An admissible $p$-word is a word satisfying (2) of Definition \ref{words} except that it can end with $\psi_{p^{\lambda}}$.  $\psi_{p^{\lambda}}$, of height $1$ and degree $2$. The degree of a $p$-word is therefore given as in (1) of \ref{words}. Notice we do not make the exception when $p=2$.
\end{definition}

In what follows we abuse notations to let words denote elements of a DGA rather then homology classes, as we did earlier. Let $\Pi=\prod_{k}\Pi_{k}$ be the decomposition of $\Pi$ into cyclic groups of order infinity or a power of a prime, and let $w_{k}$ be a generator of $\Pi_{k}$. Also we recall the the decompositions $\Pi/(p\Pi)=\prod_{i}\Pi'_{i}$ and ${_{p}\Pi}=\prod_{j}\Pi''_{j}$ as well as the generators $\{w'_{i}\}, \{w''_{j}\}$ as in Definition \ref{U(M)p} and Theorem \ref{fundamental}. We let $E(a;k)$ or $P(;k)$ denote exterior algebras or divided power algebras over $\mathbb{Z}$ generated by a single element $a$ of degree $k$, suppressing the ring of coefficients, and consider them as graded $\mathbb{Z}$-algebras. We fix a positive integer $n$, and construct a collection of DGA's.
\begin{construction}\label{construction}
\begin{enumerate}
\item
 For each $w_{k}$ of order infinity, take the DGA $E(\sigma^{n}(w_k);n)$ with trivial differential when $n$ is odd, or $P(\sigma^{n}(w_k);n)$ when $n$ is even. We denote this DGA by $A(n)_{0}$.
\item
 For each $w_{k}$ of order $p^{\lambda}$ for some prime $p$ and positive integer $\lambda$, If $n$ is odd take $E(\sigma^{n}(w_k);n)\otimes P(\sigma^{n-1}(\psi_{p^{\lambda}})(w_k); n+1)$, or $P(\sigma^{n}(w_k),n)\otimes E(\sigma^{n-1}(\psi_{p^{\lambda}})(w_k), n+1)$. In either case we define differential  $$d(\sigma^{n-1}\psi_{p^{\lambda}})(w_k))=(-1)^{n-1}p^{\lambda}\sigma^{n}(k), d(\sigma^{n}(w_k))=0$$
 when $n$ is even.
\item
 Let $\alpha'$ be an admissible $p$-word of height $n-l-1$ and degree $q$. Consider the pair of $p$-words $\sigma^{l}\varphi_{p}\alpha'$ and $\sigma^{l+1}\gamma_{p}\alpha'$. If $n$ is odd, for each $w'_{i}$ take $E(\sigma^{l+1}\gamma_{p}\alpha'(w'_i); pq+l+1)\otimes P(\sigma^{l}\varphi_{p}\alpha'(w'_i); pq+l+2)$. If $n$ is even, for each $w'_{i}$ take $P(\sigma^{l+1}\gamma_{p}\alpha'(w'_i); pq+l+1)\otimes E(\sigma^{l}\varphi_{p}\alpha'(w'_i),pq+l+2)$.
 In both cases we take differential $$d(\sigma^{l}\varphi_{p}\alpha'(w'_i))=(-1)^{n-1}p(\sigma^{l+1}\gamma_{p}\alpha'(w'_i)), d(\sigma^{l+1}\gamma_{p}\alpha'(w'_i))=0.$$
 \item
 Let $\alpha', \sigma^{l}\varphi_{p}\alpha'$ and $\sigma^{l+1}\gamma_{p}\alpha'$ be as above. If $n$ is odd, for each $w''_{j}$ take $E(\sigma^{l+1}\gamma_{p}\alpha'(w''_j); pq+l+1)\otimes P(\sigma^{l}\varphi_{p}\alpha'(w''_j); pq+l+2)$. If $n$ is even, for each $w''_{j}$ take $P(\sigma^{l+1}\gamma_{p}\alpha'(w''_j),pq+l+1)\otimes E(\sigma^{l}\varphi_{p}\alpha'(w''_j); pq+l+2)$.
 In both cases we take differential $$d(\sigma^{l}\varphi_{p}\alpha'(w''_j))=(-1)^{n-1}p(\sigma^{l+1}\gamma_{p}\alpha'(w''_j)),\  d(\sigma^{l+1}\gamma_{p}\alpha'(w''_j))=0.$$
\end{enumerate}
 We take all the DGA's constructed in (2), (3), (4) and denote their tensor product by $A(n)_{p}$. Finally we take $A(n)=A(n)_{0}\otimes_{\mathbb{Z}}\bigotimes A(n)_{p}$. It is easily seem that there is a homomorphism of divided power algebra $f: H_{*}(A(n))\rightarrow H_{*}(K(\Pi,n);\mathbb{Z})$ taking the class represented by the word $\alpha(k)$ (resp. $\alpha(i),\alpha(j)$ to the homology class given by the operations $\alpha(w_{k})$, (resp. $\alpha(w'_{i}, \alpha(w''_{j})$. The following theorem from \cite{Ca} (Cartan and Serre, 1954-1955) is stated in a more modern form than the original.
\end{construction}
\begin{theorem}[H. Cartan, Theorem 1, exp. 11, \cite{Ca}, Cartan and Serre, 1954-1955]\label{integral}
 $f: H_{*}(A(n))\rightarrow H_{*}(K(\Pi,n);\mathbb{Z})$ is an epimorphism. For any prime $p$, the restriction $f:H_{*}(A(n)_{0}\otimes A(n)_{p})\rightarrow H_{*}(K(\Pi,n);\mathbb{Z})$ is a $p$-local isomorphism.
\end{theorem}

\begin{remark}
 The kernel of $f$ is not always trivial. Indeed, when $n=2$ and $\Pi=\mathbb{Z}/2$ with a generator $w_{1}$, we take $x_{k}=\varphi_{2}\gamma_{2^{k}}\psi_{2}(w_1)$ for $k\geq 1$, $y_{k}=\sigma\gamma_{2^{k}}\psi_{2}(w_1)$ for $k\geq 0$, and $z=\sigma^{2}(w_1)$. Consider the DGA over $\mathbb{Z}$
 $$A=\bigotimes_{k\geq 1}P(x_{k}; 2^{k+1}+2)\otimes\bigotimes_{k\geq 0} E(y_{k}; 2^{k+1}+1)\otimes P(z; 2)$$ with $d(x_{k})=-2y_{k}, d(y_{k})=0$ for $k\geq 1$, $d(y_{0})=-2z$ and $d(z)=0$. In particular we have
 $$d(y_{0}\gamma_{2}(z))=-2z\gamma_{2}(z)=-6\gamma_{3}(z),$$
 which implies that the homology class $\gamma_{3}(z)$ is of order $6$, but by Theorem \ref{fundamental}, $H_{*}(K(\mathbb{Z}/2, 2); \mathbb{Z})$ has only $2$-primary elements. Hence $2\gamma_{3}z\neq 0$ is in the kernel of $f$ since it is a $3$-torsion. However, as we see in the Section \ref{DG Alg}, when $\Pi=\mathbb{Z}$ we do have $H_*(A(n))\cong H_{*}(K(\mathbb{Z},n);\mathbb{Z})$.
\end{remark}
\bibliographystyle{abbrv}
\bibliography{refJTA2019}

\begin{thebibliography}{10}

\bibitem{An2}
B.~Antieau.
\newblock On the integral {T}ate conjecture for finite fields and
  representation theory.
\newblock {\em arXiv preprint arXiv:1504.04879}, 2015.

\bibitem{An}
B.~Antieau and B.~Williams.
\newblock The topological period--index problem over 6-complexes.
\newblock {\em Journal of Topology}, 7(3):617--640, 2013.

\bibitem{An1}
B.~Antieau and B.~Williams.
\newblock The period-index problem for twisted topological {K}--theory.
\newblock {\em Geometry \& Topology}, 18(2):1115--1148, 2014.

\bibitem{Br1}
E.~H. Brown.
\newblock Twisted tensor products, {I}.
\newblock {\em Annals of Mathematics}, pages 223--246, 1959.

\bibitem{Br}
E.~H. Brown~Jr.
\newblock The cohomology of {$BSO_n$} and {$BO_n$} with integer coefficients.
\newblock {\em Proceedings of the American mathematical society}, pages
  283--288, 1982.

\bibitem{Ca}
H.~Cartan and J.-P. Serre.
\newblock {\em S{\'e}minaire {H}enri Cartan Vol. 7}.
\newblock Numdam, 1954-1955.

\bibitem{Co}
J.-L. Colliot-Th{\'e}lene.
\newblock Exposant et indice d'alg{\`e}bres simples centrales non
  ramifi{\'e}es.
\newblock {\em ENSEIGNEMENT MATHEMATIQUE}, 48(1/2):127--146, 2002.

\bibitem{Ei1}
S.~Eilenberg.
\newblock Singular homology theory.
\newblock {\em Annals of Mathematics}, pages 407--447, 1944.

\bibitem{Ei}
S.~Eilenberg and S.~M. Lane.
\newblock On the groups {$H(\Pi, n)$}, {I}.
\newblock {\em Annals of Mathematics}, 58(1):55--106, 1953.

\bibitem{Ga}
I.~Garc{\'\i}a-Etxebarria and M.~Montero.
\newblock Dai-freed anomalies in particle physics.
\newblock {\em arXiv preprint arXiv:1808.00009}, 2018.

\bibitem{Gu}
X.~Gu.
\newblock The topological period-index problem over 8-complexes, {I}.
\newblock {\em to appear in the Journal of Topology, arXiv preprint
  arXiv:1709.00787}, 2017.

\bibitem{Gu1}
X.~Gu.
\newblock The topological period-index problem over 8-complexes, {II}.
\newblock {\em arXiv preprint arXiv:1803.05100}, 2017.

\bibitem{Hs}
W.~Y. Hsiang.
\newblock {\em Cohomology theory of topological transformation groups},
  volume~85.
\newblock Springer Science \& Business Media, 2012.

\bibitem{Ka2}
M.~Kameko.
\newblock Representation theory and the cycle map of a classifying space.
\newblock {\em arXiv preprint arXiv:1601.06199}, 2016.

\bibitem{Ka}
M.~Kameko and N.~Yagita.
\newblock The brown-peterson cohomology of the classifying spaces of the
  projective unitary groups {$PU(p)$} and exceptional lie groups.
\newblock {\em Transactions of the American Mathematical Society},
  360(5):2265--2284, 2008.

\bibitem{Ko}
A.~Kono and M.~Mimura.
\newblock On the cohomology of the classifying spaces of {$PSU(4n+2)$} and
  {$PO(4n+2)$}.
\newblock {\em Publications of the Research Institute for Mathematical
  Sciences}, 10(3):691--720, 1975.

\bibitem{Ko1}
A.~Kono, M.~Mimura, N.~Shimada, et~al.
\newblock Cohomology of classifying spaces of certain associative {$H$}-spaces.
\newblock {\em Journal of Mathematics of Kyoto University}, 15(3):607--617,
  1975.

\bibitem{Mc}
J.~McCleary.
\newblock {\em A user's guide to spectral sequences}.
\newblock Number~58. Cambridge University Press, 2001.

\bibitem{Mi}
R.~J. Milgram et~al.
\newblock The bar construction and abelian {$H$}-spaces.
\newblock {\em Illinois Journal of Mathematics}, 11(2):242--250, 1967.

\bibitem{Mo}
R.~E. Mosher and M.~C. Tangora.
\newblock {\em Cohomology operations and applications in homotopy theory}.
\newblock Courier Corporation, 2008.

\bibitem{Sw}
R.~M. Switzer.
\newblock {\em Algebraic topology-homotopy and homology}.
\newblock Springer, 2017.

\bibitem{To}
H.~Toda.
\newblock Cohomology of classifying spaces.
\newblock In {\em Homotopy Theory and Related Topics}, pages 75--108, Tokyo,
  Japan, 1987. Mathematical Society of Japan.

\bibitem{Va}
A.~Vavpeti{\v{c}} and A.~Viruel.
\newblock On the mod p cohomology of {$BPU (p)$}.
\newblock {\em Transactions of the American Mathematical Society}, pages
  4517--4532, 2005.

\bibitem{Vi}
A.~Vistoli.
\newblock On the cohomology and the chow ring of the classifying space of
  {$PGL_p$}.
\newblock {\em Journal f{\"u}r die reine und angewandte Mathematik (Crelles
  Journal)}, 2007(610):181--227, 2007.

\end{thebibliography}
\end{document}